\documentclass[12pt,a4paper,dvipsnames]{article}

\usepackage{hyperref}
\usepackage{pgf}
\usepackage[utf8]{inputenc}

\usepackage{mathtools}
\usepackage{amsfonts}
\usepackage{amssymb}
\usepackage{amsthm}
\usepackage{dsfont}
\usepackage[linesnumbered]{algorithm2e}
\usepackage{setspace}
\usepackage{extarrows}

\usepackage{graphicx}
\usepackage{multicol}
\PassOptionsToPackage{dvipsnames}{xcolor}
\usepackage{tikz,xcolor}

\usetikzlibrary{shapes,arrows}
\usetikzlibrary{hobby}
\usetikzlibrary{decorations.pathreplacing}

\usepackage{caption}
\usepackage[english]{babel}		
\usepackage[T1]{fontenc}			
\usepackage{color}         			
\usepackage{listings}					

\renewcommand{\subsectionmark}[1]{}

\usepackage{geometry}
\newcommand*\dif{\mathop{}\!\mathrm{d}}

\newtheorem{theorem}{Theorem}[section]
\newtheorem{lemma}[theorem]{Lemma}
\newtheorem{conjecture}[theorem]{Conjecture}
\newtheorem{question}[theorem]{Question}
\newtheorem{proposition}[theorem]{Proposition}

\newtheorem{corollary}[theorem]{Corollary}
\newtheorem{remark}[theorem]{Remark}
\renewcommand{\P}{\mathbb{P}}
\newcommand{\N}{\mathbb{N}}

\newcommand{\E}{\mathbb{E}}

\newcommand{\etav}{\eta^\text{\normalfont v}}
\newcommand{\etac}{\eta^\text{\normalfont c}}
\newcommand{\Xv}{X^\text{\normalfont v}}
\newcommand{\Xc}{X^\text{\normalfont c}}
\newcommand{\GW}{\textup{\textbf{GW}}}
\newcommand{\AGW}{\textup{\textbf{AGW}}}
\newcommand{\UGW}{\textup{\textbf{UGW}}}
\newcommand{\EX}{\textup{EX}}
\newcommand{\jointree}{\bullet\hspace*{-3.3pt}\text{-\hspace*{1.5pt}} }
\newcommand\abs[1]{\left| #1\right|}
\newcommand{\Qu}{\mathbb{Q}^{\text{\normalfont v}}_{\rho}}
\newcommand{\Qa}{\mathbb{Q}^{\text{\normalfont c}}_{\alpha}}
\newcommand{\Pu}{\mathbb{P}^{\text{\normalfont v}}_{\rho}}
\newcommand{\Pa}{\mathbb{P}^{\text{\normalfont c}}_{\alpha}}

\newcommand{\QuT}{\mathbb{Q}^{\text{\normalfont v}}_{\rho,T}}

\newcommand{\PuT}{\mathbb{P}^{\text{\normalfont v}}_{\rho,T}}

\title{The speed of the tagged particle in the exclusion process on Galton-Watson trees}
\author{
Nina Gantert$^\ast$ and Dominik Schmid$^\ast$}
\date{\today}

\begin{document}


\maketitle

\abstract{ We study two different versions of the simple exclusion process on augmented Galton-Watson trees, the constant speed model and the varying speed model. In both cases, the simple exclusion process starts from an equilibrium distribution with non-vanishing particle density. Moreover, we assume to have initially a particle in the root, the tagged particle. We show for both models that the tagged particle has a positive linear speed and we give explicit formulas for the speeds.}

\phantom{.} \hspace{0.2cm} \textbf{Keywords:} Exclusion process, Galton-Watson trees, tagged particle, ergodicity \\
\phantom{.} \hspace{0.5cm} \textbf{AMS 2000 subject classification:} 60K35, 60K37

\let\thefootnote\relax
\footnotetext{ $^\ast$ \textit{TU München, Germany. E-Mail}: \nolinkurl{nina.gantert@tum.de}, \nolinkurl{dominik.schmid@tum.de}}

\section{Introduction}  

The simple exclusion process is a classical example of an interacting particle system which was intensively studied over the last decades, see \cite{KLO:Fluctuations,L:Book2,L:interacting-particle} for an overview.  It serves as one of the standard models to describe random movements of particles in a gas. Intuitively, it can be described as follows: For a given graph, we place indistinguishable particles on the sites of the graph such that each vertex is occupied by at most one particle. The particles then independently perform simple random walks under an exclusion rule. This means that whenever a particle would move to an occupied site, this move is suppressed. 
In this article, the underlying graph will be chosen randomly, as a supercritical, augmented Galton--Watson tree without leaves. We will consider two models, the \textbf{variable speed model} where a particle 
attempts to cross all adjacent edges with rate $1$, respectively, and the \textbf{constant speed model} where 
particles wait exponential times with expectation $1$ before choosing an adjacent edge (see below for precise definitions).
After the choice of the tree, we consider a stationary starting distribution where we condition on initially having a particle, the \textbf{tagged particle}, in the root. The variable speed model and the constant speed model have different invariant distributions.
We study the evolution of the tagged particle over time.  \\
Our motivation is two-fold. On the one hand, the tagged particle in exclusion processes has been studied for 
several graphs or several transition probabilities on the lattice, see Section \ref{sec:related}. On the $d$-dimensional lattice, if
the transition probabilities are given by
a random walk with drift, the Bernoulli product measures with parameter $\rho$ are invariant distributions and it was proved in \cite{K:CLTforSEP} and \cite{S:TaggedParticleZd}
that the speed of the tagged particle is $1-\rho$ times the speed of a single particle.
This formula for the speed is what one would naively expect: one could argue that since the density of empty sites is $1-\rho$, only a proportion of $1-\rho$ of the steps can be carried out.
In the paper \cite{CCGS:SpeedTree}, the result of \cite{S:TaggedParticleZd} 
was generalized for regular trees.
Indeed, for exclusion processes on regular trees, the Bernoulli product measures with parameter $\rho$ are again invariant distributions and the speed of the tagged particle is again $1-\rho$ times the speed of a single particle.
Exclusion processes in random environments are less studied. We consider Galton--Watson trees as an example for a random environment, where invariant measures of the exclusion process are known.
It turns out that for the variable speed model, Bernoulli product measures with parameter $\rho$ are again invariant distributions and the speed of the tagged particle is again $1-\rho$ times the speed of a single particle. However, for the constant speed model, we have invariant measures which are still product measures but not with identical marginals, and the formula for the speed can not be guessed as easily. 
On the other hand, random walks on Galton--Watson trees have been investigated intensively, we refer to the seminal paper \cite{LPP:ErgodicGalton}. We believe that it is interesting to consider tagged particles in exclusion processes as natural generalizations of random walks. One difficulty is that the position of the tagged particle is not a Markov chain, and therefore the proof of its transience is not straightforward. \\
Our formulas for the speed of the tagged particle rely on explicitly given invariant measures for the environment, seen from the tagged particle, and we show that these invariant measures are ergodic for the environment process.
The description of these invariant measures and the proof of their ergodicity are the main novelty in our paper.
There are two sets of techniques, one coming from random walks on Galton--Watson trees and 
the other one from exclusion processes. In contrast to the case of regular trees, our proofs require to intertwine the techniques from both sets.

\subsection{Outline of the paper}

This paper is organized as follows. 
In Section \ref{sec:model}, we define our model and quote some results of \cite{L:interacting-particle} on invariant measures for the variable speed exclusion process and the constant speed exclusion process on trees, respectively.
We give our main result and some remarks in Section \ref{sec:main}.
Section \ref{sec:related} contains a description of some related work.
In Section \ref{sec:SpacesofTrees}, we define a common probability space for locally finite, rooted trees and the respective exclusion processes on them. This will allow us to study the environment process in Section \ref{sec:environment}, which can be interpreted as the exclusion process ``seen from the tagged particle''.  We provide stationary measures for the environment process in both models of the simple exclusion process. The arguments in this section are based on the ideas of Lyons et al.\@ for studying random walks on Galton--Watson trees. Since the motion of the tagged particle is itself not a Markov process, a key step is to show that the tagged particle is transient. This is accomplished in Section \ref{sec:transience} by combining the results of Section \ref{sec:environment} on Galton--Watson trees and martingale techniques of Liggett \cite[Section III.4]{L:Book2} for the motion of a tagged particle in exclusion processes. A similar approach was used in \cite{CCGS:SpeedTree} to show transience of the tagged particle for regular trees. In Section \ref{sec:ergodicity}, we show ergodicity for the environment process. In contrast to the case of regular trees, this is done by intertwining different techniques coming from random walks on Galton--Watson trees and interacting particle systems, i.e.\@ we combine coupling arguments of Saada in \cite{S:TaggedParticleZd} for the exclusion process on $\mathbb{Z}^d$ with drift, with regeneration time arguments of Lyons and Peres in \cite[Chapter 17]{LP:ProbOnTrees} for the random walk on Galton--Watson trees. 
From the ergodicity of the environment process,  we deduce a law of large numbers for the position of the tagged particle in Section \ref{sec:speed}. We conclude with an outlook on related open problems.

\subsection{The model}\label{sec:model}

Let $G=(V,E)$ be a locally finite graph. Let $p: V \times V \rightarrow [0,\infty)$ be a function which satisfies $p(v,v)=0$ for all $v \in V$. The \textbf{exclusion process} with transition rates $p(\cdot,\cdot)$ will be defined as the Feller process $(\eta_t)_{t\geq 0}$ with state space $\{ 0,1\}^{V}$ generated by the closure of
\begin{equation}\label{def:generatorExclusionProcess}
\mathcal{L}f(\eta) = \sum_{x,y \in V} p(x,y) \eta(x)(1-\eta(y))\left[ f(\eta^{x,y})-f(\eta) \right] 
\end{equation} for all cylinder functions $f$. We denote by $\eta^{x,y} \in \{ 0,1\}^{V}$ the configuration where we exchange the values at positions $x$ and $y$ in $\eta\in \{ 0,1\}^{V}$. For a given configuration $\eta$, a site $x$ is \textbf{occupied} by a particle if $\eta(x)=1$ and \textbf{vacant} otherwise. Moreover, we write $\deg(x)$ for the degree of $x$. 
We consider two different ways of defining the \textbf{simple exclusion process} on $G$. For $p(x,y) = \mathds{1}_{\{ x,y \} \in E }$, we refer to the resulting process $(\etav_t)_{t \geq 0}$ as the \textbf{variable speed model}, for $p(x,y) = \deg(x)^{-1}\mathds{1}_{\{ x,y \} \in E }$, we call $(\etac_t)_{t \geq 0}$ the \textbf{constant speed model} of the simple exclusion process. 
The terms ``variable speed model'' and  ``constant speed model'' go back to \cite{BD:InvarianceUnbounded} who consider random walks among random conductances.
In words, each particle in the variable speed model at a site $x$ has an exponential waiting time with parameter $\deg(x)$ independently of all other particles. When the time is up, it jumps to one of its neighbors uniformly at random under an exclusion rule. In the constant speed model, the particles wait according to i.i.d.\ $\exp(1)$-random variables. They then choose one neighbor uniformly at random and jump to the selected site if it is vacant.
Note that the two models of the simple exclusion process agree for regular graphs up to a deterministic time change.  \\ 

In the following, the underlying graph $G$ will be given as an \textbf{augmented Galton--Watson tree} $(T,o)$ with vertex set $V(T)$, edge set $E(T)$ and a distinguished vertex $o \in V(T)$ called the \textbf{root}. More precisely, let $(p_k)$ for $k \in \N_0=\N \cup \{0\}$ be a sequence of non-negative numbers with $\sum_{k=0}^\infty p_k =1$, which defines the \textbf{offspring distribution} $\mu$ of the tree by $\mu(k)=p_k$ for all $k \in \N_0$. We construct $(T,o)$ in such a way that each site has precisely $k+1$ neighbors with probability $p_k$ for all $k\in \N_0$ independently of all other sites. To do so, define a starting vertex $o$ and recursively, starting from $o$, let every site have a number of descendants drawn independently according to $\mu$. The resulting tree is called \textbf{Galton--Watson tree}. Since in this construction, the root has on average one neighbor less than all other sites, we add one additional descendant to $o$ and apply the same recursion in order to obtain an augmented Galton--Watson tree, see Section \ref{sec:SpacesofTrees} for a precise definition. In this article, we assume that the underlying Galton--Watson branching process is supercritical and without leaves, i.e.\@ we have that 
\begin{equation}\label{eq:ExpectationFinite}
p_0=0  \qquad \text{ and } \qquad \mathfrak{m}:= \sum_{k \geq 1}k p_k\in (1,\infty)
\end{equation} holds. In particular, the corresponding augmented Galton--Watson tree is almost surely locally finite since every vertex has only a finite number of descendants. 
The following result states the existence of the exclusion process on Galton--Watson trees which is well-known in the interacting particle systems community. However we could not find an appropriate reference and therefore include a proof in the appendix.
\begin{proposition} \label{pro:existence}
For almost every Galton--Watson tree $(T,o)$, the simple exclusion processes according to the variable speed model and the constant speed model on $(T,o)$ are well-defined Feller processes. More precisely, they are Feller processes with generators given by \eqref{def:generatorExclusionProcess}.
\end{proposition}
For a given realization $(T,o)$ of an augmented Galton--Watson tree, we describe a parametrized set of invariant measures with respect to both models of the simple exclusion process. For $\rho \in [0,1]$, let $\pi_{\rho,T}$ be the Bernoulli-$\rho$-product measure on $\{0, 1\}^{V(T)}$, i.e.\@ 
\begin{equation}\label{invmeasvar}
\pi_{\rho,T}\left( \eta \colon \eta(x)= 1 \right) = \rho
\end{equation}
for all $x \in V(T)$. For $\alpha \in [0,\infty)$, let $\nu_{\alpha,T}$ denote the product measure on $\{ 0, 1 \}^{V(T)}$ with marginals
\begin{equation}\label{invmeasconst}
\nu_{\alpha,T}\left( \eta \colon \eta(x)= 1 \right) = \frac{\alpha \deg(x)}{1+\alpha\deg(x)} 
\end{equation} 
for all $x \in V(T)$. The measures $\pi_{\rho,T}$ are invariant for the simple exclusion process $(\etav_t)_{t \geq 0}$ in the variable speed model whenever $\rho \in [0,1]$, see \cite[Chapter VIII, Theorem 2.1]{L:interacting-particle}. Similarly, the measures $\nu_{\alpha,T}$ are invariant for the simple exclusion process $(\etac_t)_{t \geq 0}$ in the constant speed model for all $\alpha \in [0,\infty)$. When we condition to initially have a particle in the root, we call the resulting measures the \textbf{Palm measures} $\pi_{\rho,T}^{\ast}$ and $\nu_{\alpha,T}^{\ast}$ on $\{ 0,1 \}^{V(T)}$ given by
\begin{align*}
\pi_{\rho,T}^{\ast}( \ . \ ) &:= \pi_{\rho,T}\left( \ . \ \mid   \eta(o)=1 \right)  \\ \nu_{\alpha,T}^{\ast}( \ . \ ) &:= \nu_{\alpha,T}\left( \ . \ \mid   \eta(o)=1 \right)
\end{align*}
for all $\rho \in (0,1)$ and $\alpha \in (0, \infty)$. For $\rho=0$ and $\alpha=0$, the simple exclusion process started from $\pi_{0,T}^{\ast}$, respectively $\nu_{0,T}^{\ast}$, is the simple random walk on $T$ in the respective model starting in the root $o$. When we choose $\pi_{\rho,T}^{\ast}$, respectively $\nu_{\alpha,T}^{\ast}$, as an initial distribution of the simple exclusion process, the particle initially placed in the root is called the \textbf{tagged particle}. We denote by $(\Xv_t)_{t\geq 0}$ the position of the tagged particle in $(T,o)$ in the variable speed model  and by $(\Xc_t)_{t\geq 0}$ its position in $(T,o)$ in the constant speed model of the simple exclusion process.

\subsection{Main result}\label{sec:main}

Our main result is to establish a law of large numbers for the tagged particle in the simple exclusion process when starting from a Palm measure on an augmented Galton--Watson tree. For a rooted tree $(T,o)$ and $x \in V(T)$, we write $|x|$ for the shortest path distance from the root. 
\begin{theorem}\label{thm:LLN} Let $Z$ be distributed according to the offspring distribution $\mu$. Then for almost every augmented Galton--Watson tree $(T,o)$, the following holds:
\begin{itemize}
\item[(i)] Let $(\etav_t)_{t \geq 0}$ on $(T,o)$ have initial distribution $\pi_{\rho,T}^{\ast}$ for some $\rho \in [0,1)$. Then $(\Xv_t)_{t \geq 0}$ satisfies
\begin{equation}\label{eq:LLNvariableSpeed}
\lim_{t \rightarrow \infty} \frac{\abs{\Xv_t}}{t} = (1-\rho) \E\left[ \frac{Z-1}{Z+1}\right] \left(\E\left[ \frac{1}{Z+1}\right] \right)^{-1} 
\end{equation} almost surely.
\item[(ii)] Let $(\etac_t)_{t \geq 0}$ on $(T,o)$ have initial distribution $\nu_{\alpha,T}^{\ast}$ for some $\alpha \in [0,\infty)$. Then $(\Xc_t)_{t \geq 0}$ satisfies
\begin{equation}\label{eq:LLNConstantSpeed}
\lim_{t \rightarrow \infty} \frac{\abs{\Xc_t}}{t} =  \E\left[ \frac{Z-1}{Z+1}\frac{1}{\alpha(Z+1)+1}\right] 
\end{equation} almost surely.
\end{itemize} In particular, the tagged particle has almost surely a strictly positive speed.
\end{theorem}
If the tree is regular, i.e.\@ $ Z \equiv \mathfrak{m}$, then $1- \rho$ corresponds to $\frac{1}{\alpha(\mathfrak{m}+1)+1}$ (compare \eqref{invmeasvar} and \eqref{invmeasconst}) and indeed the two formulas \eqref{eq:LLNvariableSpeed}
and \eqref{eq:LLNConstantSpeed} agree up to the deterministic time change $\left(\E\left[ \frac{1}{Z+1}\right] \right)^{-1} $.
\begin{remark}  \begin{itemize}
\item[(i)]In the constant speed model for $\alpha \rightarrow 0$, we recover the result of Lyons et al.\@ on the speed of a random walk on supercritical Galton--Watson trees without leaves \cite{LPP:ErgodicGalton}. 
\item[(ii)] For the variable speed model, we see a linear scaling in the density $1-\rho$ of empty sites. Similar results are known for an exclusion process with drift on $\mathbb{Z}^d$ and without drift on the regular tree, see \cite{CCGS:SpeedTree,S:TaggedParticleZd}. 
\item[(iii)] In the constant speed model, we have that
\begin{equation*}
\E\left[ \frac{Z-1}{Z+1}\frac{1}{\alpha(Z+1)+1}\right]  \leq \E\left[ \frac{Z-1}{Z+1}\right]\E\left[\frac{1}{\alpha(Z+1)+1}\right] 
\end{equation*} holds, with strict inequality unless $ Z \equiv \mathfrak{m}$.
Hence, in general the scaling of the speed is lower than linear in the averaged density of empty sites.
\end{itemize}
\end{remark}

\subsection{Related work}\label{sec:related}

In the last decades, many results for random walks on Galton--Watson trees were obtained. The study of random walks on Galton--Watson trees goes back to Grimmett and Kesten who proved that the simple random walk on supercritical Galton--Watson trees conditioned on non-extinction is almost surely transient \cite{GK:NetworkCompleteGraph}. Lyons et al.\@ showed that the random walk has then almost surely a positive linear speed and calculated the velocity explicitly  \cite{LPP:ErgodicGalton}. The case of a random walk on Galton--Watson trees with bias was studied by Lyons et al.\@ in \cite{LPP:BiasedRWonGWT}. More recent treatments of the speed of random walks on Galton--Watson trees include  \cite{A:Unimodular,A:SpeedBiasGW,GMPV:RWonGW} among others. An introduction to this topic can be found in the book of Lyons and Peres \cite[Chapter 17]{LP:ProbOnTrees}. \\
Studying the behavior of the tagged particle in an exclusion process is a classical problem \cite{S:InteractionMP}. When the underlying graph is $\mathbb{Z}^d$, many results were obtained. 
For $\mathbb{Z}^d$ when the transition probabilities are symmetric and not concentrated on the nearest neighbors in the one-dimensional case, Kipnis and Varadhan established a central limit theorem for the tagged particle in their famous paper \cite{KV:Additive}. Their result was the starting point for a sequence of papers showing central limit theorems for the position of the tagged particle, see  \cite{KLO:Fluctuations,L:Book2} for an overview.  
Of course one may consider different graphs: For the $d$-dimensional ladder graph, a central limit theorem for the position of the tagged particle was proved by Zhang \cite{Z:TaggedLadder}.
In the case of translation invariant transition probabilities, a law of large numbers for the position of the  tagged particle is known in all regimes \cite{K:CLTforSEP,S:TaggedParticleZd}. If the transition probabilities on $\mathbb{Z}^d$ are translation invariant and not symmetric, 
the tagged particle has a positive linear speed. 
For the exclusion process on regular trees, Chen et al.\@ established a law of large numbers \cite{CCGS:SpeedTree}. We obtain their results as a special case.  
For random environments of the exclusion process, less results are known. In \cite{C:EPrandomEnvironment}, Chayes and Liggett investigate the invariant  distributions of the exclusion process in a one-dimensional i.i.d.\ random environment.

\section{Spaces and measures for trees}\label{sec:SpacesofTrees}

In this section, we introduce spaces and measures for rooted trees which allow us to study the simple exclusion process and locally finite, rooted trees on a common probability space. We write $(T,o) \in \mathcal{T}$ for a tree $T$ with root $o$, where $\mathcal{T}$ denotes the space of all rooted, locally finite trees. We denote by $B_r(T,o)$ the ball of radius $r$ around the root of $T$ with respect to the graph distance. We say that two rooted trees $(T,o),(T^{\prime},o^{\prime}) \in \mathcal{T}$ are \textbf{isomorphic} on a ball of radius $r$ (and write $B_r(T,o) \cong B_r(T^{\prime},o^\prime)$), if there exists a bijection $\phi \colon B_r(T,o) \rightarrow B_r(T^{\prime},o^\prime)$ such that $\phi(o)=o^\prime$ and $\{x,y\} \in E(T)$ for $x,y \in B_r(T,o)$ holds if and only if $\{\phi(x),\phi(y)\} \in E(T^\prime)$. In words, two trees are isomorphic on a ball of radius $r$ around the root when the sites of distance at most $r$ from the root can be mapped one-to-one such that the adjacency structure of the tree is preserved.
\\
The space $\mathcal{T}$ will be equipped with the \textbf{local topology}, that is the topology introduced by the distance function $\tilde{d}_{\text{\normalfont loc}}$ on $\mathcal{T}$ given by
\begin{equation*}
\tilde{d}_{\text{\normalfont loc}}((T,o),(T^{\prime},o^{\prime})) := \frac{1}{1+\tilde{R}} 
\end{equation*} for all trees $(T,o), (T^{\prime},o^{\prime}) \in \mathcal{T}$, where
\begin{equation*}
\tilde{R}= \sup\left\{ r \in \N_0 \colon B_r(T,o) \cong B_r(T^{\prime},o^\prime) \right\} \ . 
\end{equation*} In particular, the open sets on $\mathcal{T}$ will be generated by the sets of all trees which are isomorphic on a ball of radius $r$ for some $r \geq 0$. Note that $\tilde{d}_{\text{\normalfont loc}}$ is not a metric, but only pseudo-metric, since trees which are isomorphic on $B_r(T,o)$ for every $r\geq 0$ have $\tilde{d}_{\text{\normalfont loc}}$-distance $0$.
In order to turn $\mathcal{T}$ together with $\tilde{d}_{\text{\normalfont loc}}$ into a metric space, we consider isomorphism classes of trees. We say that two trees $(T,o),(T^{\prime},o^{\prime}) \in \mathcal{T}$ are \textbf{isomorphic} if $\tilde{d}_{\text{\normalfont loc}}((T,o),(T^{\prime},o^{\prime}))=0$ and write $[\mathcal{T}]$ for the set of isomorphism classes. It is a well-known result that $([\mathcal{T}],\tilde{d}_{\text{\normalfont loc}})$ forms a Polish space, see \cite{LPP:ErgodicGalton}.  \\
Let the space $\Omega$ of $0/1$-colored, locally finite, rooted trees be defined as
\begin{equation}
\Omega := \left\{ \left( T,o,\eta\right) \colon \eta \in \{ 0,1 \}^{V(T)}, (T,o) \in \mathcal{T} \right\} \ .
\end{equation} We let $B_r(T,o,\eta)$ denote the ball of radius $r$ around the root $o$ of $T$ where each site receives a color $0$ or $1$ according to $\eta$. Similar to the case of unlabeled trees, we say that $(T,o,\eta)$ and $(T^{\prime},o^{\prime},\eta^{\prime})$ are \textbf{isomorphic} on a ball of radius $r$ (and write $B_r(T,o,\eta) \cong B_r(T^{\prime},o^\prime,\eta^\prime)$) if $B_r(T,o) \cong B_r(T^{\prime},o^\prime)$ for some bijection $\phi$ as well as $\eta(v)=\eta^\prime(\phi(v))$ holds for all $v\in B_r(T,o)$. The space $\Omega$ is equipped with the topology induced by
\begin{equation*}
d_{\text{\normalfont loc}}((T,o,\eta),(T^{\prime},o^{\prime},\eta^{\prime})) := \frac{1}{1+R} 
\end{equation*} with
\begin{equation*}
R= \sup\left\{ r \in \N_0 \colon B_r(T,o,\eta) \cong B_r(T^{\prime},o^{\prime},\eta^{\prime}) \right\} 
\end{equation*} for all $ (T,o,\eta), (T^{\prime},o^\prime,\eta^{\prime})\in \Omega$. Again, we will restrict ourselves to isomorphism classes of $0/1$-colored trees in order to obtain a Polish space $([\Omega],d_{\text{\normalfont loc}})$, see Lemma 2.7 in \cite{S:ColoredBS}. \\
For a fixed tree $(T,o) \in \mathcal{T}$, we define
\begin{equation*}
\Omega_T := \left\{ (T,o,\eta) \in \Omega \colon \eta \in \{ 0,1\}^{V(T)} \right\} \subseteq \Omega
\end{equation*} to be the space of $0/1$-configurations on $(T,o)$. Moreover, let
\begin{equation}
\tilde{\Omega}_T := \left\{ (T,x,\eta) \in \Omega \colon \eta \in \{ 0,1\}^{V(T)}, x \in V(T) \right\} \subseteq \Omega
\end{equation} be the space of $0/1$-configurations on $(T,o)$ and on all (isomorphism classes of) trees obtained from 
$(T,o)$ by shifting the root.
In addition, we denote by
\begin{equation*}
\Omega^{\ast} := \left\{ (T,o,\eta) \in \Omega \colon \eta(o)=1 \right\} \subseteq \Omega
\end{equation*} the set of configurations in $\Omega$ with occupied root and define $\Omega_T^{\ast}$ and $\tilde{\Omega}_T^{\ast}$ similarly. 
Note that forming the above subspaces is consistent under taking isomorphism classes of trees, e.g.\@ we have that $[\Omega_T \cap \Omega] = [\Omega_T] \cap [\Omega]$ holds. From now on, we only work on isomorphism classes of trees and drop the brackets in the notation. Let us stress that we define all probability measures on the subspaces of $(\mathcal{T},\tilde{d}_{\text{\normalfont loc}})$ and $(\Omega,d_{\text{\normalfont loc}})$ with respect to the Borel-$\sigma$-algebra. \\
Let $\GW$ denote the \textbf{Galton--Watson measure} on $\mathcal{T}$ which is induced by the Galton--Watson branching process, see \cite[Chapter 4]{LP:ProbOnTrees}. More precisely, we define $\GW$ for families of rooted trees 
\begin{equation*}
\mathcal{T}_T(r) := \left\{ (T^{\prime},o^{\prime}) \in  \mathcal{T} \colon B_r(T^{\prime},o^{\prime})= B_r(T,o)\right\}
\end{equation*} with $r\in \N$ and $(T,o) \in  \mathcal{T}$ fixed. The measure $\GW$ assigns now to $\mathcal{T}_T(r)$ the probability that the genealogical tree of a branching process with offspring distribution $\mu$ agrees with $B_r(T,o)$ up to generation $r$. Then, the usual extension arguments yield the probability measure $\GW$ on $\mathcal{T}$. In the same way, we define $\AGW$ to be the \textbf{augmented Galton--Watson measure} on $\mathcal{T}$ by taking a branching process where the first particle has one additional child. One may also define $\AGW$ directly on $\mathcal{T}$ by choosing two independent trees according to $\GW$ and joining their roots by an edge,
see \cite[Chapter 17]{LP:ProbOnTrees}.
The simple exclusion process in the variable speed model can be seen as a process on $\Omega$ with initial distribution $\Pu$ for
\begin{equation}
\Pu := \AGW \times \pi^{\ast}_{\rho,T}
\end{equation} 
being a semi-direct product of $\AGW$ on $\mathcal{T}$ and $\pi^{\ast}_{\rho,T}$. Similarly, we refer to the simple exclusion process in the constant speed model as a process on $\Omega$ with initial distribution $\Pa$ for
\begin{equation}
\Pa := \AGW \times \nu^{\ast}_{\alpha,T}
\end{equation}
being a semi-direct product of $\AGW$ on $\mathcal{T}$ and $\nu^{\ast}_{\alpha,T}$. In particular, this construction as a semi-direct product defines $\Pu$ for $0/1$-colored balls of radius $r$. However, in contrast to $\Pu$, we can not determine $\Pa$ for $0/1$-colored balls of radius $r$ in a direct way. This is due to the fact that in order to determine the color of a vertex at distance $r$ according to $\nu^{\ast}_{\alpha,T}$, one has to know the number of its adjacent sites in distance $r+1$ from the root. To remedy this problem, we condition according to the number of children in the $(r+1)^{\text{\normalfont th}}$ generation for each site at level $r$. For the resulting balls of radius $r+1$ with colors only up to level $r$, we can now give sense to the measure $\Pa$. We conclude this section by noting that for a fixed augmented Galton--Watson tree $(T,o)\in \mathcal{T}$, the simple exclusion process on $(T,o)$ is a Feller process with values in the space $\Omega_T$ for both models. Hence, instead of working with the measures $\nu^{\ast}_{\alpha,T}$ and $\pi^{\ast}_{\rho,T}$ on a fixed Galton--Watson tree $(T,o)\in \mathcal{T}$, we will from now on study the measures $\Pa$ and $\Pu$ on the space $\Omega$ and restrict the space to $\Omega_T$ whenever we condition on a certain underlying tree $(T,o)\in \mathcal{T}$.

\section{Stationarity for the environment process}\label{sec:environment}  

In this section, we study the simple exclusion process on augmented Galton--Watson trees "seen from the tagged particle". This is a Markov process with values in $\Omega^{\ast}$
which we will call the \textbf{environment process}.
The state of the environment process at time $t$ is the $0/1$-colored tree given by the configuration of the exclusion process on the original tree whose root is shifted to the position of the tagged particle at time $t$.
Its state can change in two ways: either the coloring outside the root changes according to the exclusion process, or the root of the tree is shifted, the latter happens if and only if the tagged particle moves. See below for precise definitions.

Studying the environment process is a common approach to prove limit theorems for random walks in random environment, see \cite{Z:RWRE} for an introduction to this technique. Moreover, the  environment process can be used to show a law of large numbers for the tagged particle in the simple exclusion process on regular trees and on $\mathbb{Z}^d$ with drift, see \cite{CCGS:SpeedTree,S:TaggedParticleZd}. For the exclusion process on $\Omega$ with transition rates $p(\cdot, \cdot)$, we define the corresponding environment process to be the Feller process with state space $\Omega^{\ast}$ generated by the closure of
\begin{align}\label{def:generatorEnvironment}
Lf(T,o,\zeta) &= \sum_{x,y \neq o} p(x,y)\zeta(x)(1-\zeta(y))\left[ f(T,o,\zeta^{x,y})- f(T,o,\zeta)\right] \nonumber \\
&+ \ \sum_{z \sim o} p(o,z) (1-\zeta(z))\left[ f(T,z,\zeta^{o,z})- f(T,o,\zeta)\right] 
\end{align} for all cylinder functions $f$, where $\sim$ denotes the relation of two sites being adjacent. We write $L^{\textup{v}}$ and $L^{\textup{c}}$ for the generators of the environment process of the simple exclusion process in the variable speed model and in the constant speed model, respectively. 
Note that the generator can be split into two parts, namely into transitions which do only exchange particles and do not change the underlying tree, as well as into transitions which involve the root of the tree.
More precisely, we define the generators
\begin{align*}
(L_{\textup{ex}}^{\textup{c}}f)(T,o,\zeta)  &:=  \sum_{x,y \neq o} \frac{1}{\deg(x)}\zeta(x)(1-\zeta(y))\left[ f(T,o,\zeta^{x,y})- f(T,o,\zeta)\right] 
\end{align*} as well as
\begin{align*}
(L_{\textup{sh}}^{\textup{c}}f)(T,o,\zeta)  &:= \sum_{z \sim o} \frac{1}{\deg(o)} (1-\zeta(z))\left[ f(T,z,\zeta^{o,z})- f(T,o,\zeta)\right] 
\end{align*} for the environment process in the constant speed model for $(T,o,\zeta)\in \Omega^{\ast}$ and all cylinder functions $f$. The generators $L_{\textup{ex}}^{\textup{v}}$ and $L_{\text{sh}}^{\textup{v}}$ for the environment process in the variable speed model are defined analogously. \\

We want to investigate the invariant measures of the environment process. We provide two classes 
of reversible measures for the environment process, $\Qu$ for $\rho \in (0,1)$ and $\Qa$ for $\alpha \in (0,\infty)$, such that $\Qu$ and $\Pu$, respectively $\Qa$ and $\Pa$, are equivalent (i.e.\@ mutually absolutely continuous) for all $\rho \in (0,1)$ and $\alpha \in (0,\infty)$.
Let us stress once again that we work on isomorphism classes of trees in order to properly define stationary measures for the environment process. \\
For the environment process in the variable speed model, we will use the ideas of Aldous and Lyons \cite{A:Unimodular}. Consider the \textbf{unimodular Galton--Watson measure} $\UGW$ which we obtain from $\AGW$ by weighting a tree according to the reciprocal of the degree of its root, i.e.\@
\begin{equation}\label{def:UGWviaAGW}
\frac{\dif \UGW}{\dif \AGW} (T,o) = \left(\E\left[ \frac{1}{Z+1}\right]\right)^{-1} \cdot \frac{1}{\deg(o)} 
\end{equation} for $(T,o) \in \mathcal{T}$, where $Z$ is distributed according to the offspring distribution $\mu$. We define $\Qu$ on $\Omega^{\ast}$ to be the probability measure given as the semi-direct product 
\begin{equation}
\Qu := \UGW \times \pi^{\ast}_{\rho,T} 
\end{equation} for all $\rho\in (0,1)$. As pointed out by the authors of \cite{A:Unimodular}, the measure $\AGW$ on $\mathcal{T}$ gives the environment process a natural bias proportional to the degree of the root. This bias is compensated by the Radon--Nikodym derivative in \eqref{def:UGWviaAGW}. 
For the environment process in the constant speed model with parameter $\alpha \in (0,\infty)$, we let $\Qa$ denote the probability measure on $\Omega^{\ast}$ which is absolutely continuous with respect to $\Pa$ and satisfies
\begin{equation}\label{def:invariantEnvironment}
\frac{\dif \Qa}{\dif \Pa}(T,o,\zeta) = \left(\E\left[ \frac{1}{\alpha(Z+1)+1}\right]\right)^{-1} \cdot \frac{1}{\alpha\deg(o)+1}
\end{equation} for all $(T,o,\zeta) \in \Omega^{\ast}$, where $Z$ is distributed according to the offspring distribution $\mu$. We want to provide some intuition for the Radon--Nikodym derivative in \eqref{def:invariantEnvironment}. Observe that the semi-direct product $\AGW\times \nu_{\alpha,T}$ satisfies
\begin{equation*}
(\AGW\times \nu_{\alpha,T})\left(\deg(o)=k | \zeta(o)=1 \right) = \frac{\alpha k}{\alpha k +1} \cdot \frac{p_{k-1}}{\sum_{k \geq 1} \frac{\alpha k}{\alpha k +1} p_{k-1}}
\end{equation*} for all $k \geq 1$. Since the root is always occupied in the environment process, we expect to see a similar weighting of the degree of $o$ within $\Qa$. Recall that we have $\AGW\left( \deg(o)=k \right)= p_{k-1}$ for all $k \in \N$. Since $\AGW$ provides for the environment process a natural bias proportional to the degree of the root $o$, it remains to include the factor of $\frac{1}{\alpha k +1}$ for $\Qa$.
We now show that $\Qu$ and $\Qa$ are indeed reversible measures for the environment process for all $\rho \in (0,1)$ and $\alpha \in (0,\infty)$, respectively. For an introduction to reversibility of Feller processes, we refer to Liggett \cite[Chapter II.5]{L:interacting-particle}.  

\begin{proposition}\label{pro:invariance} Fix parameters $\rho \in (0,1)$ and $\alpha \in (0,\infty)$ for the measures $\Qu$ and $\Qa$, respectively. Then the following statements hold.
\begin{itemize}
\item[(i) \label{item:pro1}] The measure $\Qu$ is reversible for the environment process generated by $L^{\textup{v}}$. 
\item[(ii)\label{item:pro2}] The measure $\Qa$ is reversible for the environment process generated by $L^{\textup{c}}$. 
\end{itemize} In particular, the measures $\Qu$ and $\Qa$ are invariant for the environment process in the variable speed model and the constant speed model, respectively.
\end{proposition}
\begin{proof}
It suffices to show reversibility with respect to the different parts of the generators $L^{\textup{v}}$ and  $L^{\textup{c}}$. By construction, the processes on $\Omega^{\ast}$ associated to the generators $L_{\text{ex}}^{\textup{v}}$ and $L_{\text{ex}}^{\textup{c}}$, respectively, leave the underlying tree unchanged and ignore all moves involving the root. Recall that for every $(T,o) \in \mathcal{T}$, the measures $\pi_{\rho,T}$ and $\nu_{\alpha,T}$ are invariant for the simple exclusion process on $(T,o)$ for all $\rho \in (0,1)$ in the variable speed model as well as for all $\alpha \in (0,\infty)$ in the constant speed model, respectively, see \cite[Chapter VIII, Theorem 2.1]{L:interacting-particle}. Moreover, the argument used to show Theorem 2.1(b) in \cite[Chapter VIII]{L:interacting-particle} also proves reversibility of the measures $\pi_{\rho,T}$ and $\nu_{\alpha,T}$, i.e.\@ for all $(T,o) \in \mathcal{T}$ and all cylinder functions $f$ and $g$, we have that
\begin{equation*}
\int f(\eta)(\mathcal{L}^{\text{\normalfont v}}g)(\eta) \dif \pi_{\rho,T} = \int (\mathcal{L}^{\text{\normalfont v}}f)(\eta)g(\eta) \dif \pi_{\rho,T}
\end{equation*}
as well as
\begin{equation*}
\int f(\eta)(\mathcal{L}^{\text{\normalfont c}}g)(\eta) \dif \nu_{\alpha,T} = \int(\mathcal{L}^{\text{\normalfont c}}f)(\eta)g(\eta) \dif \nu_{\alpha,T}
\end{equation*} holds for all $\rho \in (0,1)$ and $\alpha \in (0,\infty)$, where $\mathcal{L}^{\text{\normalfont v}}$ and $\mathcal{L}^{\text{\normalfont c}}$ denote the generators of the variable speed model and the constant speed model of the simple exclusion process on $(T,o)$, respectively. Since the Palm measures $\pi^{\ast}_{\rho,T}$ and $\nu^{\ast}_{\alpha,T}$ have the same law as $\pi_{\rho,T}$ and $\nu_{\alpha,T}$ except at the root, this shows reversibility of the measures $\Qu$ and $\Qa$ for the processes generated by $L_{\text{ex}}^{\textup{v}}$ and $L_{\text{ex}}^{\textup{c}}$, respectively. 
We now show reversibility with respect to the processes generated by $L_{\text{sh}}^{\textup{v}}$ and $L_{\text{sh}}^{\textup{c}}$ following the ideas of Lyons et al.\@ in \cite{LPP:ErgodicGalton} for the random walk on Galton--Watson trees. For $0/1$-colored trees $(T,o,\zeta),(T^{\prime},o^{\prime},\zeta^{\prime}) \in \Omega$, let $(T \jointree T^{\prime},o,\zeta \jointree \zeta^{\prime})$ denote the tree, where we join the roots of $T$ and $T^{\prime}$ by an edge and let the resulting tree have its root at $o$. For Borel sets $C,D \subseteq \Omega$, we define 
\begin{equation*}
C \jointree D := \left\{ (T \jointree T^{\prime},o,\zeta \jointree \zeta^{\prime}) \in \Omega \  \colon (T,o,\zeta) \in C , (T^{\prime},o^{\prime},\zeta^{\prime}) \in D \right\} .
\end{equation*} 
For disjoint trees $(T_1,o_1,\zeta_1), \dots, (T_k,o_k,\zeta_k) \in \Omega$, let $\left(\bigvee_{i=1}^k T_i, o^{\prime},\bigvee_{i=1}^k  \zeta_i\right) \in \Omega^{\ast}$ denote the tree where we connect the roots to a new vertex $o^{\prime}$ forming the new root with color $1$. Similarly, we define
\begin{equation*}
\bigvee_{i=1}^k F_i := \left\lbrace \left(\bigvee_{i=1}^k  T_i, o^{\prime},\bigvee_{i=1}^k  \zeta_i \right) \colon (T_i,o_i,\zeta_i) \in F_i \right\rbrace
\end{equation*} for Borel sets $F_1, \dots, F_k \subseteq \Omega$. Note that $\bigvee_{i=1}^k F_i$ is again a Borel set of $0/1$-colored trees. Moreover, for a set of trees $F \subseteq \Omega$, we define
\begin{equation*}
\bar{F} := \left\{ (T,o,\zeta^{o}) \in \Omega \colon (T,o,\zeta) \in F  \right\}
\end{equation*}  where $\zeta^{o} \in \{ 0,1\}^{V}$ denotes the configuration in which we flip the color in $\zeta \in \{ 0,1\}^{V}$ in the root $o$. 
Observe that the processes generated by $L_{\text{sh}}^{\textup{v}}$ and $L_{\text{sh}}^{\textup{c}}$ on $\Omega^{\ast}$ give rise to transition rates
\begin{equation*}
q_{\textup{sh}}^{\textup{v}}((T,o,\zeta),B) := \abs{ \{ z \in V(T) \colon z \sim o,  (T,z,\zeta^{o,z}) \in B \} }
\end{equation*}
for the variable speed model and
\begin{equation*}
q_{\textup{sh}}^{\textup{c}}((T,o,\zeta),B) := \frac{1}{\deg(o)}\abs{\{ z \in V(T) \colon z \sim o,  (T,z,\zeta^{o,z}) \in B \} }
\end{equation*}
for the constant speed model for all $(T,o,\zeta) \in \Omega^{\ast}$, respectively. We define
\begin{align*}
q_{\textup{sh}}^{\textup{v}}(A,B) &:= \int_A q_{\textup{sh}}^{\textup{v}}((T,o,\zeta),B) \dif \Qu(T,o,\zeta) \\
q_{\textup{sh}}^{\textup{c}}(A,B) &:= \int_A  q_{\textup{sh}}^{\textup{c}}((T,o,\zeta),B) \dif \Qa(T,o,\zeta) 
\end{align*} for Borel sets $A,B \subseteq \Omega^{\ast}$. Note that it suffices to show that
\begin{align*}
q_{\textup{sh}}^{\textup{v}}(A,B) &= q_{\textup{sh}}^{\textup{v}}(B,A) \\
q_{\textup{sh}}^{\textup{c}}(A,B) &= q_{\textup{sh}}^{\textup{c}}(B,A) 
\end{align*} holds for almost all Borel sets $A,B \subseteq \Omega^{\ast}$ in order to prove reversibility. Without loss of generality, we assume that $A$ and $B$ have the form $A = C \jointree \bar{D}$ and $B = D \jointree \bar{C}$ for
\begin{align}\label{eq:TreeDecompostion}
C &= \bigvee_{i=1}^k  C_i  \ \ \text{ and } \ \ D = \bigvee_{j=1}^l  D_j 
\end{align} with integers $k,l$ such that $C,C_1,\dots,C_k,D, D_1, \dots, D_l \subseteq \Omega$ are disjoint Borel sets. More precisely, for two independent samples $(T,o),(T^{\prime},o^{\prime}) \in \mathcal{T}$ of trees according to $\GW$, we have almost surely that $\tilde{d}_{\text{\normalfont loc}}((T,o),(T^{\prime},o^{\prime})) > 0$ holds. Thus, for $\Qu$ and $\Qa$, the underlying Borel-$\sigma$-algebra on $(\Omega^{\ast},d_{\text{\normalfont loc}})$ is generated up to nullsets when taking only disjoint Borel-sets $C$ and $D$ of $0/1$-colored trees into account. A similar argument applies for all remaining pairs of sets $C,C_1,\dots,C_k,D, D_1, \dots, D_l$. A visualization of the sets $C\jointree \bar{D}$ and $D \jointree \bar{C}$ is given in Figure \ref{figureStationarity}. 
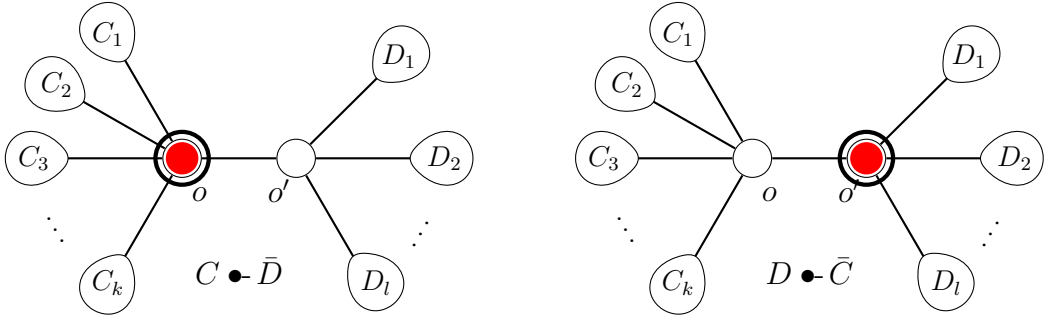
\begin{figure}
\begin{center}
\begin{tikzpicture}[scale=0.75]

	\node[shape=circle,scale=1.3,draw] (A2) at (0,0){} ;
 	\node[shape=circle,scale=1.3,draw] (B2) at (2,0){} ;
 	
 	\node[scale=1]  at (1,-2){$C\jointree \bar{D}$};
 	 	\node[scale=1]  at (11,-2){$D\jointree \bar{C}$};
 	 	
 	 	\node[scale=1]  at (0.3,-0.64){$o$};
	 	\node[scale=1]  at (1.7,-0.57){$o^{\prime}$};
	 	
	\node[shape=circle,scale=1.8,draw,line width=1.6] (A) at (0,0){} ;
	
\node[scale=0.9]  at (-2.6,0){$C_3$};
\node[scale=0.9]  at (-2.2,1.3){$C_2$};
\node[scale=0.9]  at (-1.3,2.2){$C_1$};
\node[scale=0.9]  at (-1.3,-2.2){$C_k$};

\node[scale=0.9,rotate=-60]  at (-2.2,-1.3){$\dots$};

\node[scale=0.9]  at (10-2.6,0){$C_3$};
\node[scale=0.9]  at (10-2.2,1.3){$C_2$};
\node[scale=0.9]  at (10-1.3,2.2){$C_1$};
\node[scale=0.9]  at (10-1.3,-2.2){$C_k$};

\node[scale=0.9,rotate=-60]  at (10-2.2,-1.3){$\dots$};

\node[scale=0.9]  at (4.6,0){$D_2$};
\node[scale=0.9]  at (3.8,1.8){$D_1$};
\node[scale=0.9]  at (3.4,-2.2){$D_l$};

\node[scale=0.9,rotate=60]  at (4.2,-1.3){$\dots$};

\node[scale=0.9]  at (14.6,0){$D_2$};
\node[scale=0.9]  at (13.8,1.8){$D_1$};
\node[scale=0.9]  at (13.4,-2.2){$D_l$};

\node[scale=0.9,rotate=60]  at (14.2,-1.3){$\dots$};

	\draw[thick] (A2) to (B2);		
	\draw[thick] (-2,0) to (A2);		
	\draw[thick] (-1.732,1) to (A2);		
	\draw[thick] (-1,1.732) to (A2);		
	\draw[thick] (-1,-1.732) to (A2);		
	
 \node[shape=circle,fill=red,scale=1.1] (k1) at (A2){};

\draw (0,0) to [xshift=-2cm,yshift=0cm,rotate around={0:(0,0)},closed, curve through = {(1-1.1,-0.2)(0.2-1.1,-0.4)(0-1.1,0)(0.2-1.1,0.4)(1-1.1,0.2)}] (0,0);

\draw (0,0) to [xshift=-1.732cm,yshift=1cm,rotate around={-30:(0,0)},closed, curve through ={(1-1.1,-0.2)(0.2-1.1,-0.4)(0-1.1,0)(0.2-1.1,0.4)(1-1.1,0.2)}] (0,0);

\draw (0,0) to [xshift=-1cm,yshift=1.732cm,rotate around={-60:(0,0)},closed, curve through ={(1-1.1,-0.2)(0.2-1.1,-0.4)(0-1.1,0)(0.2-1.1,0.4)(1-1.1,0.2)}] (0,0);

\draw (0,0) to [xshift=-1cm,yshift=-1.732cm,rotate around={60:(0,0)},closed, curve through = {(1-1.1,-0.2)(0.2-1.1,-0.4)(0-1.1,0)(0.2-1.1,0.4)(1-1.1,0.2)}] (0,0);

\draw (0,0) to [xshift=4cm,yshift=0cm,rotate around={180:(0,0)},closed, curve through = {(1-1.1,-0.2)(0.2-1.1,-0.4)(0-1.1,0)(0.2-1.1,0.4)(1-1.1,0.2)}] (0,0);

%

\draw (0,0) to [xshift=3.414cm,yshift=1.414cm,rotate around={225:(0,0)},closed, curve through ={(1-1.1,-0.2)(0.2-1.1,-0.4)(0-1.1,0)(0.2-1.1,0.4)(1-1.1,0.2)}] (0,0);

\draw (0,0) to [xshift=3cm,yshift=-1.732cm,rotate around={130:(0,0)},closed, curve through = {(1-1.1,-0.2)(0.2-1.1,-0.4)(0-1.1,0)(0.2-1.1,0.4)(1-1.1,0.2)}] (0,0);

	\draw[thick] (4,0) to (B2);		
	\draw[thick] (3.414,1.414) to (B2);		
	\draw[thick] (3,-1.732) to (B2);	

 	 	\node[scale=1]  at (10.3,-0.64){$o$};
	 	\node[scale=1]  at (11.7,-0.57){$o^{\prime}$};

	\node[shape=circle,scale=1.3,draw] (C2) at (10,0){} ;
 	\node[shape=circle,scale=1.3,draw] (D2) at (12,0){} ;

	\node[shape=circle,scale=1.8,draw,line width=1.6] (D) at (12,0){} ;

	\draw[thick] (C2) to (D2);		
	\draw[thick] (8,0) to (C2);		
	\draw[thick] (10-1.732,1) to (C2);		
	\draw[thick] (10-1,1.732) to (C2);		
	\draw[thick] (10-1,-1.732) to (C2);		
	
 \node[shape=circle,fill=red,scale=1.1] (k1) at (D2){};

\draw (0,0) to [xshift=8cm,yshift=0cm,rotate around={0:(0,0)},closed, curve through = {(1-1.1,-0.2)(0.2-1.1,-0.4)(0-1.1,0)(0.2-1.1,0.4)(1-1.1,0.2)}] (0,0);

\draw (0,0) to [xshift=8.268cm,yshift=1cm,rotate around={-30:(0,0)},closed, curve through ={(1-1.1,-0.2)(0.2-1.1,-0.4)(0-1.1,0)(0.2-1.1,0.4)(1-1.1,0.2)}] (0,0);

\draw (0,0) to [xshift=9cm,yshift=1.732cm,rotate around={-60:(0,0)},closed, curve through ={(1-1.1,-0.2)(0.2-1.1,-0.4)(0-1.1,0)(0.2-1.1,0.4)(1-1.1,0.2)}] (0,0);

\draw (0,0) to [xshift=9cm,yshift=-1.732cm,rotate around={60:(0,0)},closed, curve through = {(1-1.1,-0.2)(0.2-1.1,-0.4)(0-1.1,0)(0.2-1.1,0.4)(1-1.1,0.2)}] (0,0);

\draw (0,0) to [xshift=14cm,yshift=0cm,rotate around={180:(0,0)},closed, curve through = {(1-1.1,-0.2)(0.2-1.1,-0.4)(0-1.1,0)(0.2-1.1,0.4)(1-1.1,0.2)}] (0,0);

%

\draw (0,0) to [xshift=13.414cm,yshift=1.414cm,rotate around={225:(0,0)},closed, curve through ={(1-1.1,-0.2)(0.2-1.1,-0.4)(0-1.1,0)(0.2-1.1,0.4)(1-1.1,0.2)}] (0,0);

\draw (0,0) to [xshift=13cm,yshift=-1.732cm,rotate around={130:(0,0)},closed, curve through = {(1-1.1,-0.2)(0.2-1.1,-0.4)(0-1.1,0)(0.2-1.1,0.4)(1-1.1,0.2)}] (0,0);

	\draw[thick] (14,0) to (D2);		
	\draw[thick] (13.414,1.414) to (D2);		
	\draw[thick] (13,-1.732) to (D2);	


\end{tikzpicture}
\end{center}
\caption{\label{figureStationarity} Visualization of the Borel sets $C\jointree \bar{D} \subseteq \Omega^{\ast}$ and $D\jointree \bar{C} \subseteq \Omega^{\ast}$.  }
\end{figure}
In the variable speed model, we claim that
\begin{align}\label{thelongonevar}
\Qu(A) =&  (k+1)! p_k  (l!) p_l\prod_{i=1}^k(\GW\times\pi_{\rho,T})(C_i) \prod_{j=1}^l (\GW\times\pi_{\rho,T})(D_j) \nonumber\\
 & \cdot (1-\rho) \cdot \frac{1}{k+1} \cdot \left(\E\left[ \frac{1}{Z+1}\right]\right)^{-1} \, . 
\end{align} 
In particular, \eqref{thelongonevar} implies that $\Qu(A) = \Qu(B)$.
In order to show \eqref{thelongonevar}, consider a tree $ (T \jointree T^{\prime},o,\zeta \jointree \zeta^{\prime}) \in A$ with $(T,o,\zeta) \in C$ and $ (T^{\prime},o^{\prime},\zeta^{\prime}) \in  \bar{D}$ where $C$ and $D$ are given in \eqref{eq:TreeDecompostion}. By construction, the tree $(T,o,\zeta)$ must have degree $k$ at the root $o$ before the tree $ (T^{\prime},o^{\prime},\zeta^{\prime})$ is attached. There are now $(k+1)!$ ways of attaching the subtrees belonging to $C_1,\dots,C_k,D$ to the root $o$. Furthermore, the subtree $ (T^{\prime},o^{\prime},\zeta^{\prime})$  must have degree $l$ at $o^{\prime}$ before being connected to $o$ and $o^{\prime}$ has to be empty. There are $l!$ possibilities to attach to $o^{\prime}$ the trees belonging to $D_1,\dots,D_l$. By construction, under $\Qu$ the subtrees belonging to $C_1,\dots,C_k$ and $D_1,\dots,D_l$ are i.i.d.\ with law $(\GW\times\pi_{\rho,T})$. The factor $(1-\rho)$ 
is the probability that the site  $o^{\prime}$ is vacant.
Together with \eqref{def:UGWviaAGW}, this gives the above formula for $\Qu(A)$.
Since we have $$q_{\textup{sh}}^{\textup{v}}((T,o,\zeta),B) = q_{\textup{sh}}^{\textup{v}}((T^{\prime},o^{\prime},\zeta^{\prime}),A)$$ for all $(T,o,\zeta) \in A, (T^{\prime},o^{\prime},\zeta^{\prime}) \in B$, we obtain \hyperref[item:pro1]{(i)} of Proposition \ref{pro:invariance}. \\

For a given tree $(T,o) \in \mathcal{T}$ and $\alpha \in (0,\infty)$, let $\tilde{\nu}_{\alpha,T}$ be the product measure on $\{ 0,1\}^{V(T)}$ with marginals 
\begin{equation*}
\tilde{\nu}_{\alpha,T}(\eta \colon \eta(x)=1)  =  \begin{cases} \nu_{\alpha,T}(\eta \colon \eta(x)=1)  & \text{ if } x \neq o  \\ \frac{\alpha (\deg(o)+1)}{\alpha (\deg(o)+1) + 1}& \text{ if } x = o\, .\end{cases}
\end{equation*} In words, we obtain $\tilde{\nu}_{\alpha,T}$ by taking $\nu_{\alpha,T}$ except that we add $1$ to the degree of the root.
Using a similar decomposition as for $\Qu$ and taking \eqref{def:invariantEnvironment} into account, we can write $\Qa(A)$ in the constant speed model as
\begin{align}\label{thelongonecon}
\Qa(A) =& (k+1)! p_k  (l!) p_l \prod_{i=1}^k (\GW\times\tilde{\nu}_{\alpha,T})(C_i) \prod_{j=1}^l(\GW\times\tilde{\nu}_{\alpha,T})(D_j) \nonumber\\
& \cdot \frac{1}{\alpha(k+1)+1} \cdot  \frac{1}{\alpha(l+1)+1}\cdot \left(\E\left[ \frac{1}{\alpha(Z+1)+1}\right]\right)^{-1} \ . 
\end{align} 
Here, $\frac{1}{\alpha(l+1)+1}$ is the probability that  $o^{\prime}$ is vacant, and the factor $ \frac{1}{\alpha(k+1)+1}$ comes from \eqref{def:invariantEnvironment}.
In particular, \eqref{thelongonecon} implies that
\begin{equation*}
\frac{\Qa(A)}{k+1} = \frac{\Qa(B)}{l+1} 
\end{equation*} 
holds. Since we have $$ q_{\textup{sh}}^{\textup{c}}((T,o,\zeta),B) = \frac{1}{k+1} \ \text{ and } \ q_{\textup{sh}}^{\textup{c}}((T^{\prime},o^{\prime},\zeta^{\prime}),A) = \frac{1}{l+1}$$ for all $(T,o,\zeta) \in A$ and $ (T^{\prime},o^{\prime},\zeta^{\prime}) \in B$, we obtain claim \hyperref[item:pro2]{(ii)} of Proposition \ref{pro:invariance}. 
\end{proof}

\section{Transience of the tagged particle}\label{sec:transience}

Recall that the position of the tagged particle in the simple exclusion process is denoted by $(\Xv_t)_{t \geq 0}$ in the variable speed model and by $(\Xc_t)_{t \geq 0}$ in the constant speed model. Let $P_{\Pu}$ be the law of the simple exclusion process started from $\Pu$ in the variable speed model for some $\rho \in (0,1)$. Similarly, let $P_{\Pa}$ denote the law of the simple exclusion process started from $\Pa$ in the constant speed model for some $\alpha \in (0,\infty)$. For both models, we say that the tagged particle is \textbf{transient} if $(\Xv_t)_{t \geq 0}$, respectively  $(\Xc_t)_{t \geq 0}$, hits the root $P_{\Pu}$-almost surely, respectively $P_{\Pa}$-almost surely, only finitely many times.
\begin{proposition}\label{pro:transience}
The tagged particle is transient for the simple exclusion process in the variable speed model with initial distribution $\Pu$ and in the constant speed model with initial distribution $\Pa$ for all $\rho \in (0,1)$ and $\alpha \in (0,\infty)$. 
\end{proposition}

In order to show Proposition \ref{pro:transience}, we use a similar notation as introduced by Lyons et al.\@ in \cite{LPP:ErgodicGalton} for the study of random walks on Galton--Watson trees. Fix a tree $(T,o) \in \mathcal{T}$. We write $\overset{\rightarrow}{x}$ for a path $(x_0,x_1,\dots)$ in $(T,o)$ and say that it is a \textbf{ray} $\xi$ if it never backtracks, i.e.\@ we have that $x_{i} \neq x_{j}$ holds for all $i\neq j$. The \textbf{boundary} $\partial(T,o)$ of a tree $(T,o)$ is defined to be the set of all rays $\xi$ starting at the root $o$. Note that $\partial(T,o)$ consists $\AGW$-almost surely of infinitely many elements as we have $\AGW$-almost surely infinitely many sites of degree at least $3$ by our assumptions on the offspring distribution. We say that a path $\overset{\rightarrow}{x}$ \textbf{converges} to $\xi \in \partial(T,o)$ if $\xi$ is the only ray which is intersected infinitely often and $\overset{\rightarrow}{x}$ visits every site at most finitely many times. We let $[x,\xi]$ denote the unique ray starting at a site $x \in V(T)$ and converging to $\xi \in \partial(T,o)$. Note that this ray can be constructed by taking the shortest path connecting $x$ and $o$, following this path starting from $x$ until the first time a vertex of $\xi$ is hit and then following the ray $\xi$ in the direction pointing away from $o$.  For sites $x,y \in V(T)$ and a ray $\xi$, let $x\wedge_{\xi} y$ be the first site at which $[x,\xi]$ and $[y,\xi]$ meet. In particular, as $[x,\xi]$ and $[y,\xi]$ are both rays converging to $\xi$, they must agree on all but finitely many vertices with $\xi$. \\

For two sites $x,y \in V(T)$, we define their \textbf{horodistance} with respect to some given ray $\xi$ of $(T,o)$ to be the signed distance
\begin{equation}
\langle y-x \rangle_{\xi} := |y - x\wedge_{\xi} y| - |x - x\wedge_{\xi} y| \ ,
\end{equation} where $\abs{\ .\ }$ denotes the shortest path distance in $(T,o)$. We set $\langle x\rangle_{\xi} := \langle x - o \rangle_{\xi}$. In the following, we will fix for every infinite tree $(T,o) \in \mathcal{T}$ a ray $\xi=\xi(T,o) \in \partial (T,o)$ and write
\begin{equation*}
\langle y-x \rangle_{(T,o)} := \langle y-x \rangle_{\xi(T,o)}
\end{equation*} for all $x,y \in V(T)$. We require that the choice of the ray is consistent under performing shifts of the root, i.e.\@ for a given tree $(T,o) \in \mathcal{T}$, we have that
\begin{equation}\label{eq:Horodistance}
\langle y-x \rangle_{(T,z)} = \langle y \rangle_{(T,x)}
\end{equation}
holds for all $x,y,z \in V(T)$. To do so, we first choose a tree $(T,o) \in \mathcal{T}$. Let $\xi(T,o) \in \partial (T,o)$ now be an arbitrary, but fixed ray. For all $(T,z)$ with $z \in V(T)$, we then set $\xi(T,z):=[z,\xi(T,o)]$.  Now choose another tree not in this collection and iterate. Observe that for this choice of rays, the relation in \eqref{eq:Horodistance} indeed holds as the vertex $x\wedge_{\xi(T,z)} y$ remains the same for any choice of $z$, see also Figure \ref{fig:Horodistance}.
\begin{figure}

\begin{center}
\begin{tikzpicture}[scale=0.8]

\draw[draw=none,fill=gray!20] (-7,-2) rectangle (-4.5,2.5);
\draw[draw=none,fill=gray!10] (-4.5,-2) rectangle (-1.5,2.5);
\draw[draw=none,fill=gray!20] (-1.5,-2) rectangle (1.5,2.5);
\draw[draw=none,fill=gray!10] (1.5,-2) rectangle (4.5,2.5);
\draw[draw=none,fill=gray!20] (4.5,-2) rectangle (7,2.5);

	\node[shape=circle,scale=1.2,draw,fill=white] (Y2) at (-6,0){} ;
	\node[shape=circle,scale=1.2,draw,fill=white] (Z2) at (-3,0){} ;
	\node[shape=circle,scale=1.2,draw,fill=white] (A2) at (0,0){} ;
 	\node[shape=circle,scale=1.2,draw,fill=white] (B2) at (3,1){} ;
	\node[shape=circle,scale=1.2,draw,fill=white] (C2) at (3,-1) {};
	\node[shape=circle,scale=1.2,draw,fill=white] (D2) at (6,0.3){} ;
	\node[shape=circle,scale=1.2,draw,fill=white] (E2) at (6,1.7){} ;
	\node[shape=circle,scale=1.2,draw,fill=white] (F2) at (6,-1){} ;

	\node[shape=circle,scale=1.2,draw,fill=white] (G2) at (0,-1.2){} ;	
	\node[shape=circle,scale=1.2,draw,fill=white] (H2) at (0,1.2){} ;		
	
	\draw[thick] (A2) to (B2);		
	\draw[thick] (A2) to (C2);		
	\draw[thick] (B2) to (D2);		
	\draw[thick] (B2) to (E2);	
	\draw[thick] (C2) to (F2);	
	\draw[thick] (Z2) to (G2);	
	\draw[thick] (Z2) to (H2);	
	
	\draw[thick, dashed] (E2) to (7,2);			
	\draw[thick, dashed] (E2) to (7,1.5);	
	
	\draw[thick, dashed] (D2) to (7,0.3);	
	
	\draw[thick, dashed] (F2) to (7,-0.8);		
	\draw[thick, dashed] (F2) to (7,-1.2);		
	
	\draw[thick, dashed] (G2) to (1.5,-1.2);		
	\draw[thick, dashed] (H2) to (1.5,1.5);	
	\draw[thick, dashed] (H2) to (1.5,0.9);	
	
	\draw[thick, dashed] (Y2) to (-4.5,0.4);	
	\draw[thick, dashed] (Y2) to (-4.5,-0.4);	
	
	\draw[line width=2.5pt, red] (A2) to (Z2);		
	\draw[line width=2.5pt, red] (Y2) to (Z2);		
	\draw[line width=2.5pt, red, dashed] (Y2) to (-7,0);			
	
		\node[scale=0.9]  at (0,-2.4){$0$};
		\node[scale=0.9]  at (3,-2.4){$1$};
		\node[scale=0.9]  at (-3,-2.4){$-1$};
		\node[scale=0.9]  at (6,-2.4){$2$};
		\node[scale=0.9]  at (-6,-2.4){$-2$};	

		\node[scale=0.9]  at (0.3,-0.5){$o$};	
		\node[scale=1,red]  at (-4.1,0.4){$\xi$};	
				\node[scale=0.9]  at (3.05,1.55){$y$};	
						\node[scale=0.9]  at (0.05,1.7){$x$};	
								\node[scale=0.8]  at (-2.9,0.5){$x\wedge_{\xi}y$};		

\end{tikzpicture}
\end{center}
\caption{\label{fig:Horodistance}Visualization of the horodistance on a tree $(T,o)$ with ray $\xi \in \partial(T,o)$. In this example, we have that $\langle x\rangle_{(T,o)}=0$ and $\langle y-x\rangle_{(T,o)}=1$ holds.}
\end{figure}
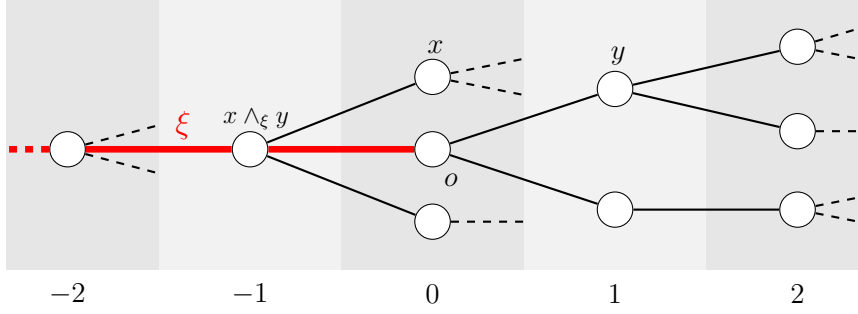
Our construction ensures that we have the same ray on all trees which one can get from a given tree by shifting its root.
For a configuration $(T,x,\zeta) \in \Omega^{\ast}$, we define the \textbf{local drift} at $x$ in the variable speed model to be
\begin{equation}
\psi^{\textup{v}}(T,x,\zeta) := \sum_{z \sim x} (1-\zeta(z)) \langle z-x\rangle_{(T,x)}  \ .
\end{equation} Similarly, the local drift at $x$ in the constant speed model is denoted by
\begin{equation}
\psi^{\textup{c}}(T,x,\zeta) := \sum_{z \sim x} \frac{1}{\deg(x)} (1-\zeta(z)) \langle z-x\rangle_{(T,x)} \ .
\end{equation}
Using these notions, we rewrite the position of the tagged particle as a martingale and a function depending only on the environment process in a ball of radius $1$ around its root. This follows the ideas of Proposition 4.1 in \cite[Chapter III]{L:Book2}.
 \begin{lemma}\label{lem:martingaleRelation} 
Fix $\rho \in (0,1)$ and $\alpha \in (0,\infty)$. Then the following two statements hold:
\begin{itemize}
\item[(i)\label{item:rel1}] Let $(T_t,o_t,\zeta_t)_{t \geq 0}$ be the environment process in the variable speed model with initial distribution $\Qu$ and natural filtration $(\mathcal{F}^{\textup{v}}_t)_{t \geq 0}$. We have that
\begin{equation}\label{eq:martingaleRelationvariable}
 \langle o_t\rangle_{(T_0,o_0)} = \int_{0}^{t} \psi^{\textup{v}}(T_s,o_s,\zeta_s) \dif s + M^{\textup{v}}_t
 \end{equation} holds for all $t \geq 0$, where $(M^{\textup{v}}_t)_{t \geq 0}$ is a martingale with respect to $(\mathcal{F}^{\textup{v}}_t)_{t \geq 0}$.
\item[(ii)\label{item:rel2}]  Let $(T_t,o_t,\zeta_t)_{t \geq 0}$ be the environment process in the constant speed model with initial distribution $\Qa$ and natural filtration $(\mathcal{F}^{\textup{c}}_t)_{t \geq 0}$. We have that
\begin{equation}\label{eq:martingaleRelationConstant}
 \langle o_t\rangle_{(T_0,o_0)} = \int_{0}^{t} \psi^{\textup{c}}(T_s,o_s,\zeta_s) \dif s + M^{\textup{c}}_t
 \end{equation} holds for all $t \geq 0$, where $(M^{\textup{c}}_t)_{t \geq 0}$ is a martingale with respect to $(\mathcal{F}^{\textup{c}}_t)_{t \geq 0}$.
\end{itemize}
 \end{lemma}
 \begin{proof} We only show part \hyperref[item:rel1]{(i)} of Lemma \ref{lem:martingaleRelation} as for part \hyperref[item:rel2]{(ii)} the same arguments apply. For a given tree $(T,o) \in \mathcal{T}$, we define 
 \begin{equation*}
 \QuT:= \Qu\left( (T^{\prime},o^{\prime}, . ) \in \cdot \  | (T^{\prime},o^{\prime})= (T,o)\right) \ .
\end{equation*} Note that we can write
 \begin{align}\label{eq:QuDecomposition} \Qu(A) &= \int_{\Omega^{\ast}} \mathds{1}_{\left\{ (T,o,\zeta) \in A \right\}} \dif \Qu(T,o,\zeta) \nonumber \\
 &= \int_{\mathcal{T}} \int_{\{ 0,1\}^{V(T)}}\mathds{1}_{\left\{ (T,o,\zeta) \in A \right\}} \dif \pi_{\rho,T}^{\ast}(\zeta) \dif \UGW(T,o) \\
 &= \int_{\mathcal{T}} \int_{\tilde{\Omega}^{\ast}_T}\mathds{1}_{\left\{ (T^{\prime},o^{\prime},\zeta) \in A \right\}} \dif \QuT(T^{\prime},o^{\prime},\zeta) \dif \UGW(T,o) \nonumber
\end{align} for all $\Qu$-measurable sets $A$. Thus, it suffices to show that 
\begin{align*}
 E_{\mathbb{Q}^{\text{\normalfont v}}_{\rho,T}}\left[\langle o_t \rangle_{(T,o)} - \langle o_s \rangle_{(T,o)} - \int_{s}^{t} \psi^{\textup{v}}(T_r,o_r,\zeta_r) \dif r \Big| \mathcal{F}^{\textup{v}}_s\right] = 0
\end{align*}
holds for all $t > s \geq 0$ and $\UGW$-almost every tree $(T,o) \in \mathcal{T}$, where $E_{\QuT}$ denotes the expectation with respect to the environment process started from $\QuT$. 
Recall that the choice of the ray for a tree in $\mathcal{T}$ is consistent under performing shifts of the root. Moreover, note that an environment process started from $\QuT$ remains in $\tilde{\Omega}^{\ast}_T$ almost surely. Using the Markov property of the environment process as well as the fact that $\Qu$ is stationary for the environment process by Proposition \ref{pro:invariance}, we see that for all $s \geq 0$,
\begin{align*}
 E_{\QuT}&\left[\langle o_t \rangle_{(T,o)}-\langle o_s \rangle_{(T,o)} - \int_{s}^t \psi^{\textup{v}}(T_r,o_r,\zeta_r) \dif r \Big|  \mathcal{F}^\textup{v}_s \right]  = \\
E_{\mathbb{Q}^{\text{\normalfont v}}_{\rho,T_s}}&\left[\langle o_{t-s} \rangle_{(T,o)}-\langle o_0 \rangle_{(T,o)} - \int_{0}^{t-s} \psi^{\textup{v}}(T_r,o_r,\zeta_r) \dif r \right] \, .
\end{align*}
Hence, it suffices to show that for $\UGW$-almost every $(T,o) \in \mathcal{T}$, we have that
\begin{equation}\label{eq:dynkinargument}
E_{\QuT}\left[\langle o_{t} \rangle_{(T,o)}-\langle o_0 \rangle_{(T,o)}\right] - \int_{0}^{t} E_{\QuT}\left[\psi^{\textup{v}}(T_s,o_s,\zeta_s)\right] \dif s = 0
\end{equation}
is satisfied for all $t \geq 0$. For a tree $(T,o) \in \mathcal{T}$, let $g$ be the function on $\tilde{\Omega}^{\ast}_T$ given by
\begin{equation*}
g(T,x,\zeta):= \langle x\rangle_{(T,o)}
\end{equation*} for all $(T,x,\zeta) \in \tilde{\Omega}^{\ast}_T$. 
Plugging $g$ into the generator in \eqref{def:generatorEnvironment}, we note that
 \begin{equation*}
L^{\textup{v}}g(T,x,\zeta) = \sum_{y \sim x}  (1- \zeta(y))\left[ g(T,y,\zeta^{x,y})- g(T,x,\zeta)\right] = \psi^{\textup{v}}(T,x,\zeta)
\end{equation*} holds for all $(T,x,\zeta) \in \tilde{\Omega}^{\ast}_T$. For the second equality, we use \eqref{eq:Horodistance} together with the relation
\begin{equation*}
\langle y-x \rangle_{(T,x)}=\langle y-x \rangle_{(T,o)} = \langle y \rangle_{(T,o)} - \langle x \rangle_{(T,o)}
\end{equation*} for all $x,y \in V(T)$ which follows from the construction of the horodistance, see Figure \ref{fig:Horodistance}. Thus, we obtain \eqref{eq:dynkinargument} by applying Dynkin's formula.
 \end{proof} 
 
\begin{proof}[Proof of Proposition \ref{pro:transience}] We will only prove transience for the tagged particle in the variable speed model of the simple exclusion process. For the tagged particle in the constant speed model of the simple exclusion process, similar arguments apply. We will show that the tagged particle has $P_{\Pu}$-almost surely a strictly positive speed with respect to the horodistance, which implies transience. Observe that the martingale $(M^{\textup{v}}_t)_{t \geq 0}$ defined via the relation \eqref{eq:martingaleRelationvariable} has stationary increments by Proposition \ref{pro:invariance} and thus satisfies a law of large numbers. Using that $\Qu$ is stationary for the environment process, we obtain by Lemma \ref{lem:martingaleRelation} that
\begin{equation*}
E_{\Qu}\left[ \lim_{t \rightarrow \infty}\frac{\langle o_t\rangle_{(T_0,o_0)}}{t} \right] =  E_{\Qu}\left[\psi^{\textup{v}}(T_0,o_0,\zeta_0)\right]
\end{equation*} holds. Moreover, from the construction of the horodistance, we see that
\begin{equation*}
\sum_{x \sim o_0} E_{\Qu}\left[\langle x \rangle_{(T_0,o_0)} | \deg(o_0)=k \right] = k-2 
\end{equation*} is satisfied for all $k\geq 2$. Thus, combining these two observations yields that
\begin{align}\label{eq:speedvariableexpected}
E_{\Qu}\left[ \lim_{t \rightarrow \infty}\frac{\langle o_t\rangle_{(T_0,o_0)}}{t} \right]  &= (1-\rho)\int_{\mathcal{T}}(\deg(o)-2) \dif \UGW(T,o) \nonumber \\
&= (1-\rho) \E\left[ \frac{Z-1}{Z+1}\right]\left(\E\left[ \frac{1}{Z+1}\right]\right)^{-1}  
\end{align} holds, where $Z$ is distributed according to the offspring distribution $\mu$. Since we have a supercritical, augmented Galton--Watson tree without leaves and $(\langle o_t\rangle_{(T_0,o_0)})_{t \geq 0}$ describes the horodistance of the tagged particle from the root within the environment process, we see that with positive $\Qu$-probability, the tagged particle has a strictly positive speed. \\
In order to show that the tagged particle is transient with respect to the initial distribution $\Pu$, suppose that there exists an initial set of $0/1$-colored trees $B \subseteq \Omega^{\ast}$ with $\Qu(B)>0$ for which the tagged particle has speed zero. Note that $B$ can be chosen such that it forms an invariant set for the environment process. Let $E_{\Qu( \cdot | B)}$ denote the expectation of the environment process started from $\Qu(\cdot | B)$. Using the arguments of the proof of Lemma \ref{lem:martingaleRelation}, the environment process must satisfy
\begin{align*}
0 &= \int_0^t E_{\Qu( \cdot| B)}\left[\psi^{\textup{v}}(T_s,o_s,\zeta_s)\right] \dif s \\
&=  \int_0^t \sum_{k \geq 2} \sum_{x \sim o_s} E_{\Qu( \cdot | B)}\left[\langle x \rangle_{(T_s,o_s)} | \deg(o_s)=k \right] \Qu(\deg(o_s)=k, \zeta_s(x)=0 | B)   \dif s  
\end{align*}  for all $t\geq 0$.  Again, from the construction of the horodistance, we see that
\begin{equation*}
\sum_{x \sim o_s} E_{\Qu( \cdot | B)}\left[\langle x \rangle_{(T_s,o_s)} | \deg(o_s)=k \right] = k-2 
\end{equation*} holds for all $k\geq 2$. Moreover, note that the conditional probability
\begin{equation}\label{eq:RestrictedSetTransience}
\Qu(\deg(o_s)=k, \zeta_s(x)=0 | B)  
\end{equation} does not depend on the particular choice of $x \sim o_s$. Hence, we see that
\begin{equation*}
0 =\int_0^t \sum_{k \geq 2} (k-2) \Qu(\deg(o_s)=k, \zeta_s(x)=0 | B)   \dif s 
\end{equation*}
must hold for all $t \geq 0$ and $x \sim o_s$. However, this gives a contradiction as the term in \eqref{eq:RestrictedSetTransience} is non-negative for all $k\geq 2$ and strictly positive for at least one $k\geq 3$. Otherwise, the underlying augmented Galton--Watson tree would almost surely be restricted to a copy of $\mathbb{Z}$. Thus, the tagged particle has a strictly positive speed $\Qu$-almost surely. We conclude since $\Pu$ and $\Qu$ are equivalent for all $\rho \in (0,1)$.
\end{proof}
\begin{corollary}\label{cor:twoEnds} Choose an initial configuration $(T,o,\zeta)\in \Omega^{\ast}$ according to $\Qu$, respectively according to $\Qa$. Consider the tagged particles within two independently sampled environment processes $(T_t,o_t,\zeta_t)_{t \geq 0}$ and $(T^{\prime}_t,o^{\prime}_t,\zeta^{\prime}_t)_{t \geq 0}$, which are both started from $(T,o,\zeta)$. Then the trajectories of the respective tagged particles converge almost surely 
 to two distinct rays $\xi,\xi^{\prime} \in \partial(T,o)$.
\end{corollary}
\begin{proof} By Proposition \ref{pro:transience}, the tagged particles in both environment processes are almost surely transient, and hence their trajectories converge almost surely to unique rays $\xi,\xi^{\prime} \in \partial (T,o)$, respectively. It remains to show that $\xi \neq \xi^{\prime}$ holds almost surely. In the proof of Proposition \ref{pro:transience}, we only require the ray of a tree to be consistent under performing shifts of the root and to be fixed at the beginning. 
Since $(T_t,o_t,\zeta_t)_{t \geq 0}$ and $(T^{\prime}_t,o^{\prime}_t,\zeta^{\prime}_t)_{t \geq 0}$ are evaluated independently, we can first sample $(T_t,o_t,\zeta_t)_{t \geq 0}$ in which the tagged particle almost surely converges to some ray $\xi \in \partial (T,o)$. For the process $(T^{\prime}_t,o^{\prime}_t,\zeta^{\prime}_t)_{t \geq 0}$, we then define the horodistance with respect to this ray $\xi$. Since almost surely
\begin{equation*}
\lim_{t \rightarrow \infty} \langle o_t\rangle_{(T_0,o_0)} = \infty
\end{equation*} holds and the tagged particle in $(T^{\prime}_t,o^{\prime}_t,\zeta^{\prime}_t)_{t \geq 0}$ converges almost surely to some ray $\xi^{\prime}\in \partial (T,o)$, the two rays $\xi$ and $\xi^{\prime}$ can almost surely not be the same.
\end{proof}

\begin{remark}\label{rem:speed} 
In \eqref{eq:speedvariableexpected}, we saw that the averaged speed of the tagged particle in the environment process in the variable speed model is given by
\begin{equation}\label{eq:speedvariable}
E_{\Qu}\left[ \lim_{t \rightarrow \infty} \frac{\langle o_t\rangle_{(T_0,o_0)}}{t} \right] = (1-\rho) \E\left[ \frac{Z-1}{Z+1}\right] \left(\E\left[ \frac{1}{Z+1}\right]\right)^{-1} 
\end{equation} for $Z \sim \mu$ and $\rho \in (0,1)$. Similarly, one derives that the averaged speed of the tagged particle in the environment process in the constant speed model is given by
 \begin{equation} \label{eq:speedConstant}
E_{\Qa}\left[ \lim_{t \rightarrow \infty} \frac{\langle o_t\rangle_{(T_0,o_0)}}{t} \right] = \E\left[ \frac{Z-1}{Z+1}\frac{1}{\alpha(Z+1)+1}\right] 
\end{equation} for $Z \sim \mu$ and $\alpha \in (0,\infty)$. We will show in Section \ref{sec:speed} that  \eqref{eq:speedvariable} and \eqref{eq:speedConstant} give the speed of the tagged particle $P_{\Pu}$-almost surely, respectively $P_{\Pa}$-almost surely, using an ergodicity argument for the environment process. 
\end{remark}
\section{Ergodicity for the environment process}\label{sec:ergodicity}
In this section, we show that the environment process started from $\Qu$ in the variable speed model and from $\Qa$ in the constant speed model, respectively, is ergodic for all $\rho \in (0,1)$ and $\alpha \in (0,\infty)$. The proof will have two main ingredients. 
First, we show that every invariant set $A$ can be represented by a set of trees, which we obtain by dropping the $0/1$-coloring in every configuration of $A$. This follows the arguments of Saada for the exclusion process on $\mathbb{Z}^{d}$ with drift, see \cite{S:TaggedParticleZd}. We then deduce ergodicity for the environment process using regeneration points. This follows the ideas of Lyons and Peres in \cite[Chapter 17]{LP:ProbOnTrees} for the simple random walk on Galton--Watson trees. 

\begin{proposition}\label{pro:ergodicity}
Fix parameters $\rho \in (0,1)$ and $\alpha \in (0,\infty)$ for the measures $\Qu$ and $\Qa$, respectively. The following two statements hold.
\begin{itemize}
\item[(i) \label{item:pro3}] The measure $\Qu$ is ergodic for the environment process generated by $L^{\textup{v}}$. 
\item[(ii)\label{item:pro4}] The measure $\Qa$ is ergodic for the environment process generated by $L^{\textup{c}}$. 
\end{itemize}
\end{proposition}
We will only show part \hyperref[item:pro3]{(i)} of Proposition \ref{pro:ergodicity}, i.e.\@ we will prove that $\Qu(A) \in \{ 0,1 \}$ holds for any set $A$ which is invariant under the environment process in the variable speed model.  For part \hyperref[item:pro4]{(ii)} of Proposition \ref{pro:ergodicity} the same arguments apply. The following lemma says that in order to determine if $(T,o,\zeta) \in A$ holds, it suffices $\Qu$-almost surely to know the underlying tree $(T,o) \in \mathcal{T}$.
\begin{lemma}\label{lem:ergodicityFixedTree} Let $A \subseteq \Omega^{\ast}$ be an invariant set for the environment process started from $\Qu$. Then for $\UGW$-almost every tree $(T,o) \in \mathcal{T}$, we have that
\begin{equation}
\int_{\tilde{\Omega}_T^{\ast}}\mathds{1}_{\left\{ (T^{\prime},o^{\prime},\zeta) \in A \right\}} \dif \QuT\left( T^{\prime},o^{\prime},\zeta\right)\in \{0,1\} 
\end{equation} holds. Moreover, we can find a Borel set of rooted trees $U \subseteq \mathcal{T}$ which is invariant under the environment process such that 
\begin{equation}
\int_{\tilde{\Omega}_T^{\ast}}\mathds{1}_{\left\{ (T^{\prime},o^{\prime},\zeta) \in A \right\}} \dif \QuT\left( T^{\prime},o^{\prime},\zeta\right) = \mathds{1}_{ \{ (T,o) \in U \} } 
\end{equation} is satisfied.
\end{lemma} In order to show Lemma \ref{lem:ergodicityFixedTree}, we follow the arguments of Saada in \cite{S:TaggedParticleZd}. A similar approach can be found in \cite{CCGS:SpeedTree} for the simple exclusion process on regular trees. A key tool in \cite{S:TaggedParticleZd} for showing ergodicity of the environment process of the simple exclusion process on $\mathbb{Z}^d$ with drift, is to use that the Bernoulli product measures are extremal invariant for the simple exclusion process on $\mathbb{Z}^d$ with drift.  \\
Similarly, our arguments are based on the fact that we have ergodicity for the simple exclusion process started from $\pi_{\rho,T}$ for $\AGW$-almost every initial tree $(T,o) \in \mathcal{T}$. More precisely, by Theorem 2.1 of \cite{J:Extremal}, we have that the measures $\pi_{\rho,T}$ are extremal invariant for the simple exclusion on $\{ 0,1\}^{V(T)}$ for $\AGW$-almost every tree $(T,o) \in \mathcal{T}$, for all $\rho \in (0,1)$. Since, for  $\AGW$-almost all trees,  the simple exclusion process on a given tree is a Markov process, this implies that the measures $\pi_{\rho,T}$ are ergodic for the simple exclusion process on a given tree $(T,o)$, see Theorem B52 of \cite{L:Book2}. \\
In order to show that the environment process on $\tilde{\Omega}_T^{\ast}$ with initial law $\QuT$ is ergodic for $\UGW$-almost every tree $(T,o) \in \mathcal{T}$, we proceed with a proof by contradiction. Suppose that for some $\rho \in (0,1)$, we have that
\begin{equation*}
0 < \QuT\left(  A  \right) < 1
\end{equation*} holds. Since the set $A$ is invariant for the environment process with starting distribution $\Qu$, it has to be invariant for the environment process on $\tilde{\Omega}_T^{\ast}$ with initial law $\QuT$ for $\UGW$-almost every tree $(T,o) \in \mathcal{T}$. Define $B := \tilde{\Omega}_T^{\ast} \setminus (A \cap \tilde{\Omega}_T^{\ast})$ and note that $B$ is a non-trivial, invariant set for the environment process started from $\QuT$. Moreover, we let the sets $\tilde{A},\tilde{B} \subseteq \Omega_T$ be given as
\begin{equation*}
\tilde{A} := \bigcup_{(\tilde{T},v,\zeta)\in A \cap \tilde{\Omega}^{\ast}_{T}  \colon (\tilde{T},v) = (T,v) } \{(T,o,\zeta)\} 
\end{equation*} and
\begin{equation*}
\tilde{B} := \bigcup_{(\tilde{T},v,\zeta)\in B   \colon (\tilde{T},v) = (T,v) } \{(T,o,\zeta)\} \ .
\end{equation*}
In words, $\tilde{A}$ is the set of all $0/1$-colorings of $(T,o)$ which we obtain by taking all $0/1$-colored trees in $A \cap \tilde{\Omega}^{\ast}_{T}$ and considering their coloring of $(T,o)$. Observe that the sets $\tilde{A}$ and $\tilde{B}$ are invariant for the simple exclusion process with initial distribution
\begin{equation}
\PuT := \delta_{(T,o)} \times\pi_{\rho,T} 
\end{equation} where $\delta_{(T,o)}$ denotes the Dirac measure on $\mathcal{T}$ concentrated on $(T,o)$. Moreover, since $\QuT$ is absolutely continuous with respect to $\PuT$ for all $\rho \in (0,1)$, we have that
\begin{equation*}
 \PuT( \tilde{A}  ) > 0 \ \text{ and } \   \PuT(\tilde{B})  >0
\end{equation*} is satisfied for $\AGW$-almost every tree $(T,o) \in \mathcal{T}$. Using the ergodicity of the simple exclusion process on $\Omega_T$, we conclude that
\begin{equation}\label{eq:setsABtilde}
 \PuT( \tilde{A} ) = \PuT( \tilde{B} ) =1 \ .
 \end{equation}
In particular, the sets $\tilde{A}$ and $\tilde{B}$ are not disjoint. From this, we want to deduce that $A$ and $B$ are not disjoint as well. 
For a tree $(T,o) \in \mathcal{T}$ and $x,y \in V(T)$, let $[x,y]$ be the sites in the shortest path connecting $x$ and $y$ in $(T,o)$. 
The following lemma is the analogue of Lemma 4 in \cite{S:TaggedParticleZd} and Lemma 2.3 in \cite{CCGS:SpeedTree}. 

\begin{lemma}\label{lem:ergodicity} For $\PuT$-almost every $(T,o,\eta)\in \Omega_T$, there exist sites $v,w,x,y,z \in V(T)$ with the following properties:
\begin{itemize}
\item[(i)\label{item1}] $(T,y,\eta) \in A$, $(T,z,\eta) \in B$
\item[(ii)\label{item2}] $\eta(v)=\eta(w)=0$ and $\eta(a)=0$ for all $a\in [y,z] \setminus \{y,z\}$
\item[(iii)\label{item3}] $x$ and $z$ are located in different branches with respect to $y$ in $T$.
\item[(iv)\label{item4}] $v,w,y$ are located in pairwise different branches with respect to $x$ in $T$.
\item[(v)\label{item5}] The path $[v,x]$ contains at least $|x-y|+1$ vacant sites.
\end{itemize}
\end{lemma}

\begin{proof} Using \eqref{eq:setsABtilde}, there almost surely exist sites $y,z \in V(T)$ such that $(T,y,\eta) \in A$ and $(T,z,\eta) \in B$ holds. Without loss of generality, the shortest path connecting $y$ and $z$ can be assumed to consist only of vacant sites. To see this, observe that for any $(T,a,\eta) \in \tilde{\Omega}_T^{\ast}$ with $\eta(a)=1$ and $a\in V(T)$, we have that either $(T,a,\eta) \in A$ or $(T,a,\eta) \in B$ holds. Hence, along the shortest path connecting $y$ and $z$, there must be a pair of occupied sites $y^{\prime}$ and $z^{\prime}$ with
$(T,y^{\prime},\eta) \in A$ and $(T,z^{\prime},\eta) \in B$
 and only vacant sites on the shortest path between them. Thus, we can take these sites $y^{\prime}$ and $z^{\prime}$ as our new choices of $y$ and $z$ which satisfy \hyperref[item1]{(i)}. By our assumptions on the augmented Galton--Watson tree, there almost surely exists a site $x$ in a branch of $y$ different from the one containing $z$ with degree at least $3$. Let $C(x,y)$ and $D(x,y)$ denote the vertices of two distinct branches of $x$ which do not intersect the path $[x,y]$. Using a Borel-Cantelli argument, we see that $C(x,y)$ and $D(x,y)$ both contain $\Pu$-almost surely a ray starting at $x$ with infinitely many vacant sites. Let $w$ be the first vacant site along that ray in $C(x,y)$. Let $v$ be the first vacant site along that ray in $D(x,y)$ such that there are $|x-y|+1$ empty sites along the path $[x,v]$.
\end{proof}

\begin{proof}[Proof of Lemma \ref{lem:ergodicityFixedTree}]
Take a configuration $(T,o,\eta)\in \Omega_T$ according to $\PuT$ which satisfies the properties \hyperref[item1]{(i)} to \hyperref[item5]{(v)} of Lemma \ref{lem:ergodicity} with sites $v,w,x,y,z$ and set
\begin{equation}\label{def:setN}
N:= \left\lbrace v,w,x,y,z, [v,x], [w,x], [x,y], [y,z] \right\rbrace \ .
\end{equation}  We fix a time $t_0 > 0$ and define a $0/1$-coloring $\tilde{\eta}\in \{ 0,1\}^{V(T)}$ as follows. We let $\tilde{\eta}$ agree with $\eta$ on $N$. On $V(T) \setminus N$, let $\tilde{\eta}$ have the law of a simple exclusion process at time $t_0$ started from $\eta$ where all moves involving a site in $N$ are suppressed. \\
We will now provide two ways of transforming $\eta$ into $\eta^{w,z}$ which only involve the sites in $N$ as depicted in Figure \ref{figureErgodicity}. This also provides two ways of changing $\tilde{\eta}$ into $\tilde{\eta}^{w,z}$ for any fixed $t_0>0$. 
At the beginning, we assume for both transformations that all particles in $[x,y]\setminus \{ y\}$ are moved into the empty sites within $[v,x]\setminus \{ v,x\}$ in an arbitrary way using only nearest neighbor moves within $N$. In a next step, the two transformations differ in performing the following transitions.
\begin{itemize}
\item[(a)\label{def:way1}] Push the particles along the path $[y,v]$ towards $v$, i.e.\@  for $\{ v_i, 1\leq i \leq k\}$ being successive vertices in $[v,y]$ with $\eta(v_i)=1$,  move the particle from $v_1$ to $v$, then the particle from $v_{2}$ to $v_{1}$ and so on. Next, we push the particles in $[z,w]$ towards $w$ in the same way. Afterwards, push the particles along $[v,y]$ towards $y$.
\item[(b)\label{def:way2}]  Push the particles along the path $[y,w]$ towards $w$ in the same way as described in \hyperref[def:way1]{(a)}. Afterwards, move the particle at $z$ to $y$ along $[z,y]$. 
\end{itemize}
At the end, in both transformations all particles which were moved to the empty sites in $[v,x]\setminus \{ x\}$ at the beginning are moved back to their original positions. \\
\begin{figure}

\begin{center}
\begin{tikzpicture}[scale=0.9,radius=1]

	\node[shape=circle,scale=1.2,draw] (A2) at (0,0.5){} ;
 	\node[shape=circle,scale=1.2,draw] (B2) at (1.5,1.5){} ;
	\node[shape=circle,scale=1.2,draw] (C2) at (3,0.5) {};
	\node[shape=circle,scale=1.2,draw] (D2) at (1.5,3.15){} ;
	\node[shape=circle,scale=1.2,draw] (E2) at (1.5,4.8){} ;
	
\node[scale=1.1]  at (1.8,4.4){$z$};
\node[scale=1.1]  at (1.5,0.9){$x$};
\node[scale=1.1]  at (0.4,0.1){$v$};
\node[scale=1.1]  at (1.8,2.7){$y$};
\node[scale=1.1]  at (2.6,0.1){$w$};

	\draw[thick] (A2) to (B2);		
	\draw[thick] (B2) to (C2);		
	\draw[thick] (B2) to (D2);		
	\draw[thick] (D2) to (E2);		

	\draw[thick] (A2) to (0,0);	
	\draw[thick] (C2) to (3,0);	
	\draw[thick] (A2) to (-0.45,0.8);	
	\draw[thick] (C2) to (3.45,0.8);		
	\draw[thick] (E2) to (1.95,5.1);	
	\draw[thick] (E2) to (1.05,5.1);

 \node[shape=circle,   fill=RoyalBlue] (k2) at (D2){}; 
 
 	\node[shape=circle,scale=1,draw] (A3) at (5,3.1){} ;
 	\node[shape=circle,scale=1,draw] (D3) at (6,4.7){} ;
 	 	\node[shape=circle,scale=1,draw] (E3) at (6,5.7){} ;
	\node[shape=circle,scale=1,draw] (C3) at (7,3.1) {};
	\node[shape=circle,scale=1,draw] (B3) at (6,3.65){} ;

	\draw[thick] (A3) to (B3);		
	\draw[thick] (B3) to (C3);		
	\draw[thick] (B3) to (D3);		
		\draw[thick] (E3) to (D3);	
	\draw[thick] (A3) to (5,2.75);	
	\draw[thick] (C3) to (7,2.75);	
	\draw[thick] (A3) to (4.7,3.3);	
	\draw[thick] (C3) to (7.3,3.3);		
	\draw[thick] (E3) to (6.3,5.9);	
	\draw[thick] (E3) to (5.7,5.9);

\begin{scope}[shift={(E3)}]
\clip (0,0) circle (0.17);
\node[shape=circle,scale=0.8,fill=red] (k1) at (E3){};
 \foreach \i in {-8,...,8}
  \foreach \j in {-8,...,8}
\fill[red!5] (1/20*\j cm,1/20*\i cm) circle (0.018);
   \end{scope}
 \node[shape=circle,scale=0.8,  fill=RoyalBlue] (k1) at (A3){};
 
  	\node[shape=circle,scale=1,draw] (A3) at (8.5,3.1){} ;
  	\node[shape=circle,scale=1,draw] (E3) at (9.5,5.75){} ;
 	\node[shape=circle,scale=1,draw] (D3) at (9.5,4.7){} ;
	\node[shape=circle,scale=1,draw] (C3) at (10.5,3.1) {};
	\node[shape=circle,scale=1,draw] (B3) at (9.5,3.65){} ;

	\draw[thick] (A3) to (B3);		
	\draw[thick] (B3) to (C3);		
	\draw[thick] (B3) to (D3);		
    \draw[thick] (D3) to (E3);	
	
	\draw[thick] (A3) to (8.5,2.75);	
	\draw[thick] (C3) to (10.5,2.75);	
	\draw[thick] (A3) to (8.2,3.3);	
	\draw[thick] (C3) to (10.8,3.3);		
	\draw[thick] (E3) to (9.8,5.9);	
	\draw[thick] (E3) to (9.2,5.9);

\begin{scope}[shift={(C3)}]
\clip (0,0) circle (0.17);
\node[shape=circle,scale=0.8,fill=red] (k1) at (C3){};
 \foreach \i in {-8,...,8}
  \foreach \j in {-8,...,8}
\fill[red!5] (1/20*\j cm,1/20*\i cm) circle (0.018);
   \end{scope}

 \node[shape=circle,scale=0.8,   fill=RoyalBlue] (k1) at (A3){};
 
   	\node[shape=circle,scale=1,draw] (A3) at (12,3.1){} ;
 	\node[shape=circle,scale=1,draw] (D3) at (13,4.7){} ;
 	\node[shape=circle,scale=1,draw] (E3) at (13,5.7){} ;
	\node[shape=circle,scale=1,draw] (C3) at (14,3.1) {};
	\node[shape=circle,scale=1,draw] (B3) at (13,3.65){} ;

	\draw[thick] (A3) to (B3);		
	\draw[thick] (B3) to (C3);		
	\draw[thick] (B3) to (D3);		
	\draw[thick] (D3) to (E3);	
	\draw[thick] (A3) to (12,2.75);	
	\draw[thick] (C3) to (14,2.75);	
	\draw[thick] (A3) to (11.7,3.3);	
	\draw[thick] (C3) to (14.3,3.3);		
	\draw[thick] (E3) to (13.3,5.9);	
	\draw[thick] (E3) to (12.7,5.9);	

\begin{scope}[shift={(C3)}]
\clip (0,0) circle (0.17);
\node[shape=circle,scale=0.8,fill=red] (k1) at (C3){};
 \foreach \i in {-8,...,8}
  \foreach \j in {-8,...,8}
\fill[red!5] (1/20*\j cm,1/20*\i cm) circle (0.018);
   \end{scope}
 \node[shape=circle,scale=0.8,   fill=RoyalBlue] (k1) at (D3){};

 \node[shape=circle,scale=1,draw] (A3) at (5,-0.55){} ;
 	\node[shape=circle,scale=1,draw] (D3) at (6,1.05){} ;
 	 	\node[shape=circle,scale=1,draw] (E3) at (6,2.05){} ;
	\node[shape=circle,scale=1,draw] (C3) at (7,-0.55) {};
	\node[shape=circle,scale=1,draw] (B3) at (6,0){} ;

	\draw[thick] (A3) to (B3);		
	\draw[thick] (B3) to (C3);		
	\draw[thick] (B3) to (D3);	
		\draw[thick] (E3) to (D3);

	\draw[thick] (A3) to (5,1.75-2.65);	
	\draw[thick] (C3) to (7,1.75-2.65);	
	\draw[thick] (A3) to (4.7,2.3-2.65);	
	\draw[thick] (C3) to (7.3,2.3-2.65);		
	\draw[thick] (E3) to (6.3,4.9-2.65);	
	\draw[thick] (E3) to (5.7,4.9-2.65);	

\begin{scope}[shift={(E3)}]
\clip (0,0) circle (0.17);
\node[shape=circle,scale=0.8,fill=red] (k1) at (E3){};
 \foreach \i in {-8,...,8}
  \foreach \j in {-8,...,8}
\fill[red!5] (1/20*\j cm,1/20*\i cm) circle (0.018);
   \end{scope}
 \node[shape=circle,scale=0.8, fill=RoyalBlue] (k1) at (C3){};
 
  \node[shape=circle,scale=1,draw] (A3) at (8.5,-0.55){} ;
 	\node[shape=circle,scale=1,draw] (D3) at (9.5,1.05){} ;
 	 	\node[shape=circle,scale=1,draw] (E3) at (9.5,2.05){} ;
	\node[shape=circle,scale=1,draw] (C3) at (10.5,-0.55) {};
	\node[shape=circle,scale=1,draw] (B3) at (9.5,0){} ;

	\draw[thick] (A3) to (B3);		
	\draw[thick] (B3) to (C3);		
	\draw[thick] (B3) to (D3);		
		\draw[thick] (E3) to (D3);	
	\draw[thick] (A3) to (8.5,1.75-2.65);	
	\draw[thick] (C3) to (10.5,1.75-2.65);	
	\draw[thick] (A3) to (8.2,2.3-2.65);	
	\draw[thick] (C3) to (10.8,2.3-2.65);		
	\draw[thick] (E3) to (9.8,4.9-2.65);	
	\draw[thick] (E3) to (9.2,4.9-2.65);

\begin{scope}[shift={(D3)}]
\clip (0,0) circle (0.17);
\node[shape=circle,scale=0.8,fill=red] (k1) at (D3){};
 \foreach \i in {-8,...,8}
  \foreach \j in {-8,...,8}
\fill[red!5] (1/20*\j cm,1/20*\i cm) circle (0.018);
   \end{scope}
 \node[shape=circle,scale=0.8,fill=RoyalBlue] (k1) at (C3){};

\node[scale=2.3] (arrowing) at (4.2,4.15){$\rightarrow$};
\node[scale=2.3] (arrowing) at (4.2,0.5){$\rightarrow$};

\node[scale=1.1] (arrowing) at (4,4.55){\hyperref[def:way1]{(a)}};
\node[scale=1.1] (arrowing) at (4,0.9){\hyperref[def:way2]{(b)}};

\node[scale=2.3] (arrowing) at (7.75,4.15){$\rightarrow$};
\node[scale=2.3] (arrowing) at (7.75,0.5){$\rightarrow$};

\node[scale=2.3] (arrowing) at (11.25,4.15){$\rightarrow$};

\begin{scope}[shift={(1.5,4.8)},scale=1.47]
\clip (0,0) circle (0.14);
\node[shape=circle,scale=1,fill=red] (k5) at (0,0){};
 \foreach \i in {-8,...,8}
  \foreach \j in {-8,...,8}
\fill[red!5] (1/20*\j cm,1/20*\i cm) circle (0.018);
   \end{scope}


\end{tikzpicture}
\end{center}
\caption{\label{figureErgodicity} Transformations for $\eta$ to $\eta^{w,z}$ in a sample of an augmented Galton--Watson tree for the special case when $v,w$ and $y$ are neighbors of $x$ and $[x,w]$ as well as $[x,v]$ are empty.}
\end{figure}
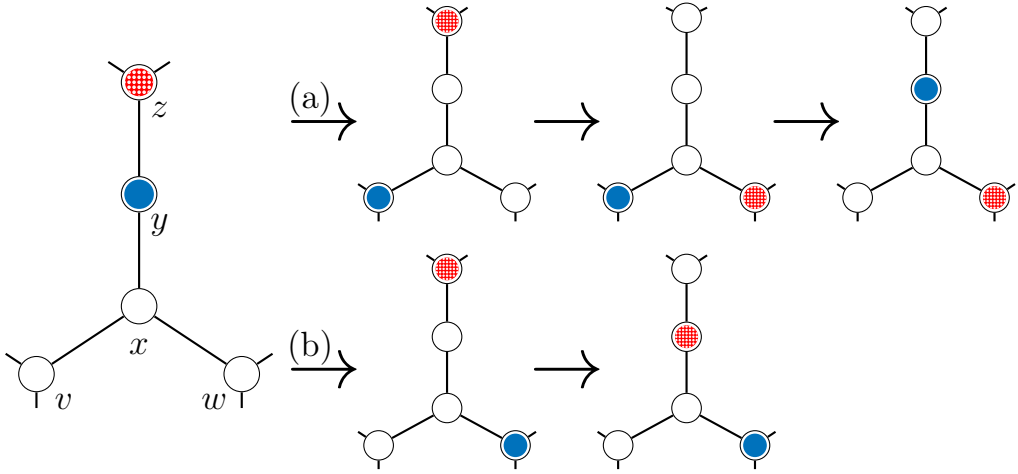

Following the transformation according to \hyperref[def:way1]{(a)}, we see that $(T,y,\tilde{\eta}^{w,z}) \in A$ holds $\PuT$-almost surely, using that $A$ is invariant for the environment process and  \hyperref[item1]{(i)} of Lemma \ref{lem:ergodicity}. For the transformation according to \hyperref[def:way2]{(b)}, note that $(T,y,\tilde{\eta}^{w,z}) \in B$ holds following the trajectory of the particle originally at $z$ and using that $B$ is invariant for the environment process. Observe that at time $t_0$, the simple exclusion process started from $(T,o,\eta)$ agrees with $(T,o,\tilde{\eta}^{w,z})$ with positive probability using the graphical representation. Hence, we obtain the desired contradiction of $A$ and $B$ not being disjoint. \\

For the second statement in Lemma \ref{lem:ergodicityFixedTree}, we let $S$ denote the set of trees which we obtain by deleting all $0/1$-colorings in the elements of $A$, i.e.\@
\begin{equation}\label{def:uncoloredTrees}
U:= \left\lbrace (T,o) \colon (T,o,\zeta) \in A\right\rbrace \subseteq \mathcal{T} \ .
\end{equation} From the construction of the $\sigma$-algebras on $\Omega$ and $\mathcal{T}$, we see that $U$ forms a Borel set of trees. Using the first statement of Lemma \ref{lem:ergodicityFixedTree}, we obtain that $U$ is invariant for the environment process started from $\Qu$. 
\end{proof} 
Next, we show that $\UGW(U) \in \{0,1\}$ holds for the set $U$ defined in \eqref{def:uncoloredTrees}. This will yield Proposition \ref{pro:ergodicity} since we have that
\begin{equation*} \Qu(A) = \int_{\mathcal{T}} \mathds{1}_{\left\{ (T,o) \in U \right\}}  \dif \UGW(T,o) 
\end{equation*} holds by Lemma \ref{lem:ergodicityFixedTree}. We follow the arguments in the proof of Theorem 17.13 in \cite{LP:ProbOnTrees} which were used in order to establish ergodicity for the environment process of the simple random walk on supercritical Galton--Watson trees.

\begin{lemma}\label{lem:mixingTreeComponent}
Let $(T_t,o_t,\zeta_t)_{t \geq 0}$ denote the environment process with state space $\Omega^{\ast}$ and initial distribution $\Qu$. The corresponding dynamical system is mixing in the tree-component, i.e.\@ we have that
\begin{equation*}
\UGW( (T_0,o_o) \in C, (T_t,o_t) \in D ) \overset{t \rightarrow \infty}{\longrightarrow}\UGW( (T_0,o_o) \in C) \UGW((T_t,o_t) \in D )
\end{equation*} holds for all Borel-sets $C,D \subseteq \mathcal{T}$. In particular, we have that $\UGW(\tilde{U}) \in \{ 0,1\}$ holds for any set of trees $\tilde{U}$ which is invariant for the environment process.
\end{lemma}
To prove Lemma \ref{lem:mixingTreeComponent}, we will need some preliminaries. Recall that the $\sigma$-algebras on $\mathcal{T}$ and $\Omega$ are generated by sets of trees which agree within a ball of finite radius around the root. Including all finite unions and intersections of these balls, we see that the balls generating the $\sigma$-algebra on $\mathcal{T}$ form a semi-algebra. Hence, using a well-known result from ergodic theory, it suffices to show mixing in the tree component for the sets $C$ and $D$ which take into account only a finite range of the tree around its root, see \cite[Exercise 2.7.3(1)]{EW:Ergodic}.  \\

For a given tree $(T,o) \in \mathcal{T}$, let $\overset{\leftrightarrow}{x}= ( \dots, x_{-1},x_0,x_{1}, \dots )$ be a bi-infinite path in $(T,o)$ with $x_0=o$. We denote by $\overset{\leftrightarrow}{T}$ the space of all such bi-infinite paths for which both ends converge to distinct rays in $\partial (T,o)$. Define the \textbf{path space} of trees to be
\begin{equation}
\text{PathsInTrees} := \left\lbrace (\overset{\leftrightarrow}{x},T) \colon \overset{\leftrightarrow}{x} \in \overset{\leftrightarrow}{T}, (T,o)\in \mathcal{T} \right\rbrace \ .
\end{equation}  Let $S$ be the map which shifts $\overset{\leftrightarrow}{x}$ to the right and changes the root of $T$ to $x_1$, i.e.\@
\begin{equation*}
(S\overset{\leftrightarrow}{x})_n = x_{n+1}
\end{equation*} holds for all $n \in \mathbb{Z}$ and
\begin{equation*}
 S(\overset{\leftrightarrow}{x},T) := (S\overset{\leftrightarrow}{x},T) \ .
\end{equation*}

As in Corollary \ref{cor:twoEnds}, choose an initial configuration $(T,o,\zeta)\in \Omega^{\ast}$ according to $\Qu$. We consider two independently sampled environment processes $(T_t,o_t,\zeta_t)_{t \geq 0}$ and $(T^{\prime}_t,o^{\prime}_t,\zeta^{\prime}_t)_{t \geq 0}$, which are both started from $(T,o,\zeta)$. Let $\overset{\rightarrow}{x}=(x_0,x_1,\dots)$ and $\overset{\rightarrow}{x}^{\prime}=(x^{\prime}_0,x^{\prime}_1,\dots)$ denote the trajectories of the tagged particles in the environment processes $(T_t,o_t,\zeta_t)_{t \geq 0}$ and $(T^{\prime}_t,o^{\prime}_t,\zeta^{\prime}_t)_{t \geq 0}$, respectively. We join the trajectories $\overset{\rightarrow}{x}$ and $\overset{\rightarrow}{x}^{\prime}$ to obtain a bi-infinite path $\overset{\leftrightarrow}{x}$ by
\begin{equation*}
\overset{\leftrightarrow}{x} := (\dots, x_2^{\prime},x^{\prime}_1, x_0,x_1,x_2,\dots )
\end{equation*}
and denote the corresponding law of $(\overset{\leftrightarrow}{x},T)$ in $\text{PathsInTrees}$ by $\EX\times\Qu$.
Note that $\overset{\leftrightarrow}{x} \in \overset{\leftrightarrow}{T}$ holds almost surely, since by Corollary \ref{cor:twoEnds}, the two trajectories of the tagged particles converge almost surely to two distinct rays. 
 The path space is equipped with the $\sigma$-algebra $\mathcal{F}$ induced by the environment processes.  Since the environment process is a reversible Markov process with respect to $\Qu$, observe that
\begin{equation*}
\left( \text{PathsInTrees}, \mathcal{F}, \EX\times\Qu, S\right) 
\end{equation*}
forms a measure-preserving system, i.e.\@ we have that
\begin{equation}\label{eq:shiftInvariance}
\EX\times\Qu(F) = \EX\times\Qu(S^{-1}F)
\end{equation} holds for all $F \in \mathcal{F}$. Define the event of having a \textbf{regeneration point} at $x_0$ to be
\begin{equation*}
\textup{Regen} := \left\lbrace  (\overset{\leftrightarrow}{x},T) \in  \text{PathsInTrees} \text{ s.t. } \forall n \leq 0 \colon x_n \neq x_1 \text{ and } \forall n \geq 1 \colon x_n \neq x_{0} \right\rbrace \ .
\end{equation*}
In words, $x_0$ is a regeneration point if the edge $\{ x_0,x_1\}$ is traversed precisely once. The following lemma is the analogue of Proposition 17.12 in \cite{LP:ProbOnTrees} for the simple random walk on Galton--Watson trees.
\begin{lemma}\label{lem:RegenerationPoints} For $\EX\times\Qu$ almost every configuration $(\overset{\leftrightarrow}{x},T)$, we have infinitely many $n\in \mathbb{Z}$ such that $S^n(\overset{\leftrightarrow}{x},T) \in \textup{Regen}$ holds.
\end{lemma}

In order to show Lemma \ref{lem:RegenerationPoints}, observe that the event of having a regeneration point at $x_0$ can be written as the intersection of the event of having a \textbf{fresh point} at $x_0$
\begin{equation*}
\textup{Fresh} := \left\lbrace  (\overset{\leftrightarrow}{x},T) \in  \text{PathsInTrees} \text{ s.t. } \forall n \leq 0 \colon x_n \neq x_1  \right\rbrace 
\end{equation*}  and the event of having an \textbf{exit point} at $x_0$
\begin{equation*}
\textup{Exit} := \left\lbrace  (\overset{\leftrightarrow}{x},T) \in  \text{PathsInTrees} \text{ s.t. }  \forall n \geq 1 \colon x_n \neq x_{0} \right\rbrace \ .
\end{equation*} Using reversibility of the environment process with respect to $\Qu$ together with \eqref{eq:shiftInvariance}, we see that
\begin{equation}\label{eq:probabilityFreshExit}
\EX\times\Qu(\textup{Fresh})=\EX\times\Qu(\textup{Exit})
\end{equation} holds. Using the transience of the tagged particle together with Corollary \ref{cor:twoEnds}, we have $\EX\times\Qu$-almost surely infinitely many fresh points, i.e.\@ 
\begin{equation*}
\EX\times\Qu \left( \exists n\geq m \text{ s.t. } S^n(\overset{\leftrightarrow}{x},T) \in \textup{Fresh} \right) = 1
\end{equation*} holds for any $m\geq 0$. Since the probability of the event of having a fresh point at $x_0$ is invariant under shifts according to $S$, we conclude that the probabilities in \eqref{eq:probabilityFreshExit} must be strictly positive. Moreover, this shows that we have $\EX\times\Qu$-almost surely infinitely many exit points. For $m,n \in \mathbb{Z}$ with $m\leq n$, we define the event 
 \begin{equation*}
 H_{m,n} := \left\{S^m(\overset{\leftrightarrow}{x},T) \in \textup{Fresh}, S^n(\overset{\leftrightarrow}{x},T) \in \textup{Exit}, [x_{m},x_{n}] \cap \{ x_i, i \in \mathbb{Z}\setminus [m,n] \} = \emptyset\right\}  \ .
 \end{equation*} In words, $ H_{m,n}$ is the event that $x_m$ is a fresh point, $x_n$ is an exit point and the shortest path connecting $x_m$ and $x_n$ does not intersect the remaining trajectory.
\begin{lemma}\label{lem:Hset} There exists a $k \geq 0$ such that $\EX\times\Qu(H_{0,k}) > 0$ holds.
\end{lemma} 
 \begin{proof}
Observe that for $\EX\times\Qu$-almost every $(\overset{\leftrightarrow}{x},T) \in \textup{PathsInTrees}$, the tagged particles in the corresponding environment processes converge to distinct rays $\xi_1,\xi_2 \in \partial (T,x_0)$. Let $a\in V(T)$ be the last common vertex of $\xi_1$ and $\xi_2$. Using transience, we observe that $a$ is hit almost surely only finitely often. We choose $m$ such that $
S^m(\overset{\leftrightarrow}{x},T) \in \textup{Fresh}$ with $a \notin \{ \dots, x_{m-1},x_m\}$ and $n$ such that $S^n(\overset{\leftrightarrow}{x},T) \in \textup{Exit}$ with $a \notin \{ x_{n},x_{n+1},\dots\}$. For these choices of $m$ and $n$, we have that $(\overset{\leftrightarrow}{x},T) \in H_{m,n}$. Note that we find such $m$ and $n$ for $\EX\times\Qu$-almost every element of $\textup{PathsInTrees}$. Hence, $\EX\times\Qu(H_{m,n}) > 0$ must hold for some deterministic choice of $m$ and $n$. Set $k=n-m$ and use the fact that we have a measure-preserving system to conclude. 
 \end{proof}

 \begin{figure}

\begin{center}
\begin{tikzpicture}[scale=0.74]


\def\z{0.3}

\path[draw,use Hobby shortcut,closed=true,fill=gray!10]
(2,2) .. (3,3) .. (3,5) .. (0.5,5) .. (1,4) .. (0.8,3); 

\foreach \x in {5,10,15} {

	
	\node[shape=circle,scale=1.2,draw] (A2) at (0+\x,0.5){} ;
 	\node[shape=circle,scale=1.2,draw] (B2) at (1.5+\x,1.5){} ;
	\node[shape=circle,scale=1.2,draw] (C2) at (3+\x,0.5) {};
	\node[shape=circle,scale=1.2,draw] (D2) at (1.5+\x,2.7){} ;
	\node[shape=circle,scale=1.2,draw] (DE2) at (1.5+\x,3.9){} ;
	\node[shape=circle,scale=1.2,draw] (DF2) at (2.7+\x,3.9){} ;
	\node[shape=circle,scale=1.2,draw] (E2) at (1.5+\x,4.8+\z){} ;
	

	\draw[thick] (A2) to (0.5+\x,0.5+0.333333);	
	\draw[thick,dotted] (0.5+\x,0.5+0.333333) to (1+\x,0.5+0.666666);
    \draw[thick] (1+\x,0.5+0.666666) to (B2);				
	\draw[thick] (B2) to (C2);		
	\draw[thick] (B2) to (D2);		
	\draw[thick] (D2) to (DE2);
	\draw[thick] (DF2) to (DE2);			
	\draw[thick] (DE2) to (E2);

	\draw[thick] (A2) to (0+\x,0);	
	\draw[thick] (C2) to (3+\x,0);	
	\draw[thick] (A2) to (-0.45+\x,0.8);	
	\draw[thick] (C2) to (3.45+\x,0.8);		
	\draw[thick] (DF2) to (3.05+\x,4.2);
	\draw[thick] (DF2) to (3.05+\x,3.6);		
	\draw[thick] (E2) to (1.95+\x,5.1+\z);	
	\draw[thick] (E2) to (1.05+\x,5.1+\z);

}
\node[scale=1]  at (2.7,3.1){$N$};

\node[scale=1]  at (2,4.4+\z){$x_0$};
\node[scale=1]  at (1.5,0.9){$x$};
\node[scale=1]  at (0.4,0.1){$v$};
\node[scale=1]  at (2,2.4){$x_k$};
\node[scale=1]  at (2.6,0.1){$w$};

\begin{scope}[shift={(3+5,0.5)},scale=1.47]
\clip (0,0) circle (0.17);
\node[shape=circle,scale=1,fill=red] (k7) at (0,0){};
 \foreach \i in {-8,...,8}
  \foreach \j in {-8,...,8}
\fill[red!5] (1/20*\j cm,1/20*\i cm) circle (0.018);
   \end{scope}

 \node[shape=circle,   fill=RoyalBlue] (k4) at (0+5,0.5){};

 \begin{scope}[shift={(3+10,0.5)},scale=1.47]
\clip (0,0) circle (0.17);
\node[shape=circle,scale=1,fill=red] (k7) at (0,0){};
 \foreach \i in {-8,...,8}
  \foreach \j in {-8,...,8}
\fill[red!5] (1/20*\j cm,1/20*\i cm) circle (0.018);
   \end{scope}

 \node[shape=circle,   fill=RoyalBlue] (k6) at (1.5+10,4.8+\z){}; 
 
 \node[scale=2.3] (arrowing) at (4.2,2.5){$\rightarrow$};
 \node[scale=2.3] (arrowing) at (4.2+5,2.5){$\rightarrow$};
 \node[scale=2.3] (arrowing) at (4.2+10,2.5){$\rightarrow$};

\begin{scope}[shift={(1.5+15,2.7)},scale=1.47]
\clip (0,0) circle (0.17);
\node[shape=circle,scale=1,fill=red] (k7) at (0,0){};
 \foreach \i in {-8,...,8}
  \foreach \j in {-8,...,8}
\fill[red!5] (1/20*\j cm,1/20*\i cm) circle (0.018);
   \end{scope}

 \node[shape=circle,   fill=RoyalBlue] (k8) at (1.5+15,4.8+\z){};

	\node[shape=circle,scale=1.2,draw] (A2) at (0,0.5){} ;
 	\node[shape=circle,scale=1.2,draw] (B2) at (1.5,1.5){} ;
	\node[shape=circle,scale=1.2,draw] (C2) at (3,0.5) {};
	\node[shape=circle,scale=1.2,draw] (D2) at (1.5,3.15-0.45){} ;
	\node[shape=circle,scale=1.2,draw] (DE2) at (1.5,3.9){} ;
		\node[shape=circle,scale=1.2,draw] (DF2) at (2.7,3.9){} ;
	\node[shape=circle,scale=1.2,draw] (E2) at (1.5,4.8+\z){} ; 
 
	\draw[thick] (A2) to (0.5,0.5+0.333333);	
	\draw[thick,dotted] (0.5,0.5+0.333333) to (1,0.5+0.666666);
    \draw[thick] (1,0.5+0.666666) to (B2);		
	\draw[thick] (B2) to (C2);		
	\draw[thick] (B2) to (D2);		
	\draw[thick] (D2) to (DE2);		
	\draw[thick] (DE2) to (E2);

	\draw[thick] (A2) to (0,0);	
	\draw[thick] (C2) to (3,0);	
	\draw[thick] (A2) to (-0.45,0.8);	
	\draw[thick] (C2) to (3.45,0.8);		
	\draw[thick] (DE2) to (2,3.9);	
	\draw[thick] (E2) to (1.95,5.1+\z);	
		\draw[thick] (DF2) to (DE2);		
		\draw[thick] (DF2) to (3.05,4.2);
	\draw[thick] (DF2) to (3.05,3.6);

\draw[->,thick] (1.5,7+\z) to  [curve through={(1.8,6.3+\z) (1.3,5.6+\z)}] (1.4,5.1+\z);
\draw[->,thick] (1.5,2.84-0.45) to  [curve through={(1.2,2.5-0.45) (0.8,2.4-0.45)}] (0,1.3);
\node[scale=1]  at (0,2.5){$\overset{\leftrightarrow}{x}$};

\begin{scope}[shift={(1.5,4.8+\z)},scale=1.47]
\clip (0,0) circle (0.17);
\node[shape=circle,scale=1,fill=red] (k1) at (0,0){};
 \foreach \i in {-8,...,8}
  \foreach \j in {-8,...,8}
\fill[red!5] (1/20*\j cm,1/20*\i cm) circle (0.018);
   \end{scope}

 \node[shape=circle,  fill=RoyalBlue] (k2) at (1.5,3.9){};

\end{tikzpicture}
\end{center}
\caption{\label{figureMixing}Construction of a regeneration point at $x_0$  in a sample of an augmented Galton--Watson tree  when $w$ and $x_k$ are neighbors of $x$ and $[x,v]$ as well as $[x,w]$ are empty. The tagged particle is drawn scattered in red. }
\end{figure}
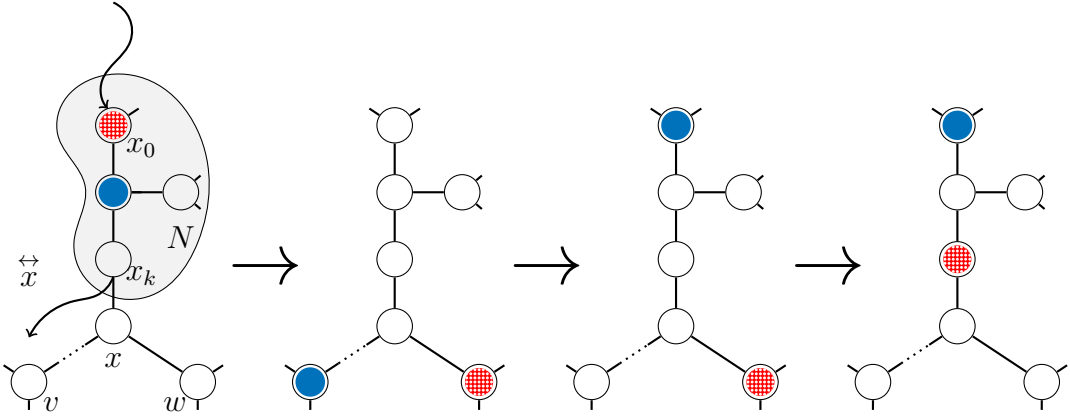

For a given configuration $(\overset{\leftrightarrow}{x},T) \in H_{0,k}$ with $k\geq 0$ of Lemma \ref{lem:Hset}, let $(T_t,o_t,\zeta_t)_{t \geq 0}$ be the underlying environment process in positive time direction. We recursively define a sequence of almost surely finite stopping times by $\tau_0:=0$ and
\begin{equation*}
\tau_{i} := \inf\{ t \geq  \tau_{i-1} \colon o_t \neq o_{\tau_{i-1}} \} 
\end{equation*} for all $i \geq 1$. We show that $(\overset{\leftrightarrow}{x},T) $ is contained in the set of regeneration points at $x_0$ with positive probability using a similar argument as in the proof of Lemma \ref{lem:ergodicityFixedTree}. 
Let $(T^{\prime}_t,o^{\prime}_t,\zeta^{\prime}_t)_{t\geq 0}$ be the environment process with the same initial configuration as $(T_t,o_t,\zeta_t)_{t\geq 0}$ but where all moves involving the tagged particle are suppressed. Using the graphical representation, note that $\zeta^{\prime}_{\tau_{k}}$ and $\zeta_{\tau_{k}}$ differ almost surely in at most finitely many values. Let $N$ denote the sites in the minimal spanning tree consisting of $\{x_0,x_1,\dots,x_k\}$ and all sites in which $\zeta^{\prime}_{\tau_{k}}$ and $\zeta_{\tau_{k}}$ differ. The proof of the following lemma uses similar arguments as the proof of Lemma \ref{lem:ergodicity}.
\begin{lemma}\label{lem:ShiftingParticlesMixing} For almost every configuration $(\overset{\leftrightarrow}{x},T)$ and $0/1$-colorings $\zeta^{\prime}_{\tau_{k}}$ and $\zeta_{\tau_{k}}$ differing only in sites within $N$, there exist $v,w,x \in V(T)$ with the following properties:
\begin{itemize}
\item[(i)\label{item1lem} ] $v,w,x \notin N$, $x_0 \notin [x,x_k]$ 
\item[(ii)\label{item2lem}] $\zeta_{\tau_{k}}(w)=0$.
\item[(iii)\label{item3lem}] $x_k,v$ and $w$ are located in pairwise different branches with respect to $x$ in $T$.
\item[(iv)\label{item4lem}] The path $[x,v]$ contains at least $|x-x_k|+k+1$ vacant sites.
\end{itemize}
\end{lemma}
\begin{proof} Let $C(x_k)$ be a branch of $x_k$ which does not contain $x_0$. Using a Borel-Cantelli argument, there almost surely exists a site $x\in C(x_k)$ with $\deg(x)\geq 3$ which is not contained in the set $N$. Let $C(x)$ and $D(x)$ be two different branches of $x$ which are disjoint of $[x,x_k]$. Note that $C(x)$ and $D(x)$ are disjoint from the set $N$ and contain almost surely an infinite number of vacant sites. Let $w$ be the first site in $C(x)$ which is empty. Similarly, let $v$ be the first site in $D(x)$ such that condition \hyperref[item4lem]{(iv)} holds.
\end{proof}

\begin{proof}[Proof of Lemma \ref{lem:RegenerationPoints}] If Lemma \ref{lem:Hset} holds for $k=0$, we conclude by applying Poincar\'e's recurrence theorem. For $k\geq 1$, we will use Lemma \ref{lem:ShiftingParticlesMixing} to provide a way of transforming $\zeta^{\prime}_{\tau_{k}}$ into $\zeta_{\tau_{k}}$ by finitely many transitions, see Figure \ref{figureMixing} for a visualization. In this transformation, the tagged particle will not come back to $x_0$ once it has left its starting point. Let $N_k := N \cup [x_k,v] \cup [x_k,w]$ for $v,w$ from Lemma \ref{lem:ShiftingParticlesMixing}. We start by moving all particles on the sites $[x_0,x] \setminus \{ x_0\}$ into empty positions in $[x,v]\setminus \{ x\}$ using only nearest neighbor transitions in $N_k$ which do not involve $x_0$. In a next step, we push the tagged particle from $x_0$ towards $w$ along the path $[x_0,w]$, i.e.\@ for $\{ w_i, 1\leq i \leq n\}$ being successive vertices in $[w,x_0]$ with $\eta(w_i)=1$,  move the particle from $w_1$ to $w$, then the particle from $w_{2}$ to $w_{1}$ and so on. In particular, note that after this push, the tagged particle is contained in $[x,w]\setminus \{ x\}$. Next, we perform nearest neighbor moves involving only the sites $N \cup [x_k,v]$ such that $z \in N \setminus [x_k,x]$ is occupied if and only if $\zeta_{\tau_{k}}(z)=1$ holds and all sites in $[x_k,x]$ are  empty. We now push the particles along the path $[w,x_k]$ towards $x_k$ in the same way as described before. Note that at this point, the tagged particle is located in $x_k$ and the constructed configuration  may differ from $\zeta_{\tau_{k}}$ only at sites $[x_k,v] \setminus \{ x_k\}$. Since the number of particles in $N_k$ is preserved, we can now perform nearest neighbor moves using only the particles in 
$[x_k,v] \setminus \{ x_k\}$ to obtain the configuration $\zeta_{\tau_{k}}$. 
Note that for almost every pair of configurations $\zeta^{\prime}_{\tau_{k}}$ and $\zeta_{\tau_{k}}$, this provides a way of transforming $\zeta^{\prime}_{\tau_{k}}$ into $\zeta_{\tau_{k}}$ modifying the exclusion process only between times $0$ and $\tau_k$ on an almost surely finite set of vertices $N_k$. Thus, under the measure $\EX\times\Qu(\ . \ | H_{0,k})$, with a positive probability all transitions among the sites $N_k$ follow precisely the above described transformation from $\zeta_0$ to $\zeta_{\tau_{k}}$ between times $0$ and $\tau_k$. Since in this case, we have by construction a regeneration point at $x_0$, we conclude that
\begin{equation*}
\EX\times\Qu(\textup{Regen} | H_{0,k})>0 
\end{equation*} holds. Using Lemma \ref{lem:Hset}, we conclude by applying Poincar\'e's recurrence theorem.  
\end{proof} 

\begin{proof}[Proof of Lemma \ref{lem:mixingTreeComponent}] We follow similar arguments as in the proof of Proposition 17.12 in \cite{LP:ProbOnTrees}. 
For a tree $(T,o) \in \mathcal{T}$ and $x \in V(T)$, let $T^x$ denote the subtree of $(T,o)$ rooted at $x$ (it contains the sites which become disconnected from $o$ when $x$ is removed). For $(\overset{\leftrightarrow}{x},T) \in \textup{Regen}$, let the \textbf{first return time} $n_{\text{Regen}}$ be given as
\begin{equation*}
n_{\text{Regen}}(\overset{\leftrightarrow}{x},T) := \inf \left\lbrace n > 0 \colon S^n(\overset{\leftrightarrow}{x},T) \in \textup{Regen}\right\rbrace 
\end{equation*}  and note that $n_{\text{Regen}}$ is almost surely finite. For $n=n_{\text{Regen}}$, we define the associated \textbf{slab} to be
\begin{equation*}
\textup{Slab}(\overset{\leftrightarrow}{x},T) := \left( \left(x_0, \dots, x_{n-1}\right),T \setminus \left( T^{x_{-1}} \cup T^{x_{n}}\right)\right) 
\end{equation*} and set $S_{\text{Regen}}:= S^{n_\text{Regen}}$. This yields an i.i.d.\ sequence $\left(\textup{Slab}(S^k_{\text{Regen}}(\overset{\leftrightarrow}{x},T))\right)_{k \in \mathbb{Z}}$ generating $(\overset{\leftrightarrow}{x},T)$. Recall that we have to verify mixing in the tree-component only for Borel-sets $C,D \subseteq T$ which take into account a finite range of the tree around the root. Since the tagged particle is transient, we see that for all $t$ sufficiently large, the events $\{(T_0,o_0) \in C \}$ and $\{(T_t,o_t) \in D \}$ depend on disjoint sets of slabs. This gives Lemma \ref{lem:mixingTreeComponent} and hence also Proposition \ref{pro:ergodicity}.
\end{proof}

\section{Speed of the tagged particle}\label{sec:speed}

Combining the results of the previous sections, we have all ingredients to prove Theorem \ref{thm:LLN}. As pointed out in Remark \ref{rem:speed}, we will use the arguments given in Section \ref{sec:transience} for showing transience of the tagged particle in order to determine the almost-sure speed of the tagged particle with respect to $P_{\Pu}$ and $P_{\Pa}$.
Recall from Lemma \ref{lem:martingaleRelation} that we can rewrite the horodistance of the tagged particle in terms of the environment process in a ball of radius $1$ around its root and a martingale.  Using the results of Sections \ref{sec:environment} and \ref{sec:ergodicity}, we obtain the following lemma as an analogue of Corollaries 4.5 and 4.16 in \cite[Chapter III]{L:Book2}.

\begin{lemma} \label{lem:martingaleErgodicity} For any $\rho \in (0,1)$ in the variable speed model and for any $\alpha \in (0,\infty)$ in the constant speed model of the simple exclusion process, the martingales $(M^{\textup{v}}_t)_{t \geq 0}$ and $(M^{\textup{c}}_t)_{t \geq 0}$ in Lemma \ref{lem:martingaleRelation} have stationary and ergodic increments.
 \end{lemma} 
 \begin{proof} We only prove the case of the variable speed model. For each $t\geq 0$, the random variable $\langle o_t \rangle_{(T_0,o_0)}$ can be expressed as a function $F_t$ of $\{(T_s,o_s,\zeta_s), 0 \leq s \leq t\}$ by following the shifts of the root, i.e.\@ it holds that 
 \begin{equation*}
\langle o_t \rangle_{(T_0,o_0)} - \langle o_0 \rangle_{(T_0,o_0)} = F_t((T_s,o_s,\zeta_s), 0 \leq s \leq t)  \ .
\end{equation*} Since the environment process is a stationary process when starting from $\Qu$, we have that
 \begin{equation*}
 M^{\textup{v}}_t - M^{\textup{v}}_s = F_{t-s}((T_r,o_r,\zeta_r), s \leq r \leq t) + \int_{s}^{t} \psi^{\textup{v}}(T_r,o_r,\zeta_r) \dif r
 \end{equation*}
holds for all $s<t$. From Propositions \ref{pro:invariance} and \ref{pro:ergodicity} we know that $\Qu$ is a stationary and ergodic measure for the environment process and so the claimed statement follows. 
 \end{proof}
 
 \begin{proof}[Proof of Theorem \ref{thm:LLN}] Using Proposition \ref{pro:ergodicity} and Lemma \ref{lem:martingaleErgodicity}, we can apply the ergodic theorem for both terms on the right-hand side of \eqref{eq:martingaleRelationvariable}, respectively, to see that
 \begin{equation*}
 \lim_{t \rightarrow \infty} \frac{\langle o_t\rangle_{(T_0,o_0)} }{t} =  (1-\rho) \E\left[ \frac{Z-1}{Z+1}\right] \left(\E\left[ \frac{1}{Z+1}\right]\right)^{-1} 
 \end{equation*} holds almost surely for $\Qu$-almost every initial configuration in the variable speed model and $\rho \in (0,1)$. Similarly, 
  \begin{equation*}
 \lim_{t \rightarrow \infty} \frac{\langle o_t\rangle_{(T_0,o_0)} }{t} =  \E\left[ \frac{Z-1}{Z+1}\frac{1}{\alpha(Z+1)+1}\right] 
 \end{equation*} holds almost surely for $\Qa$-almost every initial configuration in the constant speed model and $\alpha > 0$. Recall that the measures $\Qu$ and $\Pu$, respectively $\Qa$ and $\Pa$, are equivalent for all $\rho \in (0,1)$ and $\alpha \in (0,\infty)$. Since $(o_t)_{t \geq 0}$ describes the position of the tagged particle within the environment process, we conclude that
\begin{equation*}
 \lim_{t \rightarrow \infty} \frac{\langle \Xv_t\rangle_{(T_0,o_0)} }{t} =  (1-\rho) \E\left[ \frac{Z-1}{Z+1}\right] \left(\E\left[ \frac{1}{Z+1}\right]\right)^{-1} 
 \end{equation*} holds $P_{\Pu}$-almost surely for $\Pu$-almost every initial configuration in the variable speed model and all $\rho \in (0,1)$. Similarly, 
\begin{equation*}
 \lim_{t \rightarrow \infty} \frac{\langle \Xc_t\rangle_{(T_0,o_0)} }{t} =  \E\left[ \frac{Z-1}{Z+1}\frac{1}{\alpha(Z+1)+1}\right] 
 \end{equation*} holds $P_{\Pa}$-almost surely for $\Pa$-almost every initial configuration in the constant speed model and all $\alpha \in (0,\infty)$. 
Note that in both models of the simple exclusion process, the tagged particle converges almost surely to a ray $\xi^{\prime} \in \partial(T_0,o_0)$ different from $\xi = \xi(T_0,o_0)$. Let $a$ denote the last common vertex of $\xi$ and $\xi^{\prime}$ in the variable speed model and observe that
 \begin{equation*}
|\Xv_t| = \langle \Xv_t \rangle_{(T_0,o_0)} + 2 |a| 
\end{equation*} holds for all $t\geq 0$ sufficiently large. A similar statement is true for the tagged particle in the constant speed model.  We conclude since $|a|$ does not depend on $t$.
\end{proof}

\section{Open problems} In this article, we consider the speed of a tagged particle when the particles perform simple random walks under an exclusion rule on augmented Galton--Watson trees without leaves. However, it is a natural question to consider Galton--Watson trees which may also die out.
\begin{conjecture} On supercritical Galton--Watson trees conditioned on survival, the tagged particles in the constant speed model and in the variable speed model of the simple exclusion process have almost surely a positive linear speed.
\end{conjecture}
Another extension of our model is to consider random walks with different transition probabilities.
\begin{question} What is the speed of a tagged particle when the particles perform a biased random walk on augmented Galton--Watson trees under an exclusion rule?
\end{question} 
A classical problem for exclusion processes is the question if the tagged particle satisfies a central limit theorem. In the case where the augmented Galton--Watson tree is a $d$-regular tree with $d\geq 3$, a central limit theorem holds, see \cite[Theorem 1.3]{CCGS:SpeedTree}.
\begin{conjecture} For any $\rho \in (0,1)$, there exists a constant $\sigma_{\textup{v}}=\sigma_{\textup{v}}(\mu,\rho) \in (0,\infty)$ such that on almost every supercritical augmented Galton--Watson tree without leaves, the tagged particle $(\Xv_t)_{t \geq 0}$ in the variable speed model satisfies
\begin{equation*}
\frac{|\Xv_t|- E_{\Pu}[|\Xv_t|]}{\sqrt{t}} \overset{(d)}{\longrightarrow} \mathcal{N}(0,\sigma_{\textup{v}}^2) \ .
\end{equation*} Similarly, for every $\alpha \in (0,\infty)$, there exists a constant $\sigma_{\textup{c}}=\sigma_{\textup{c}}(\mu,\alpha) \in (0,\infty)$ such that on almost every supercritical augmented Galton--Watson tree without leaves, the tagged particle $(\Xc_t)_{t \geq 0}$ in the constant speed model satisfies
\begin{equation*}
\frac{|\Xc_t|- E_{\Pa}[|\Xc_t|]}{\sqrt{t}} \overset{(d)}{\longrightarrow} \mathcal{N}(0,\sigma_{\textup{c}}^2) \ .
\end{equation*}

\end{conjecture}

\appendix 
\section{Appendix}

In this section, we show that the simple exclusion process on Galton--Watson trees exists and is a Feller process. 
For a general introduction to Feller processes, we refer to Liggett \cite{L:interacting-particle}. In particular, Theorem 3.9 in \cite[Chapter I]{L:interacting-particle} shows that simple exclusion processes on graphs with uniformly bounded degree and uniformly bounded transition rates are Feller processes. We will now consider a more general set of underlying graphs which will include Galton--Watson trees.
Let $G=(V,E,o)$ be an infinite, locally finite, connected, rooted graph and let $p_G \in [0,1]$ denote the critical value for bond percolation on $G$, i.e.\@
\begin{equation*}
p_G := \sup\left\{ p\geq 0 \colon P_p( o \text{ is contained in an infinite open cluster})= 0 \right\}
\end{equation*} where $P_p$ denotes the law of Bernoulli bond percolation on $G$ with parameter $p$. 
\begin{proposition}\label{pro:PercolationExistence} Let $G$ be such that $p_G>0$ holds. Assume that the transition rates $\{p(x,y)\}_{x,y \in V}$ are only strictly positive for nearest neighbors $x,y$ and uniformly bounded from above by some constant $C>0$. Then the exclusion process on $G$ exists and is a Feller process with generator given in \eqref{def:generatorExclusionProcess}.
\end{proposition}
\begin{proof} We start by defining the exclusion process as a Markov process $(\eta_t)_{t \geq 0}$ with state space $\{ 0,1 \}^V$. Recall the graphical representation of the exclusion process. To every (directed) edge $(x,y)$, we assign exponential clocks with parameter $p(x,y)$. Whenever a clock rings and the site $x$ is occupied, we move the particle from $x$ to $y$, provided that $y$ is empty; otherwise nothing happens. Set $\tau=\frac{p_G}{3C}$ and note that $P(X \leq \tau)< p_{G}$ holds for $X$ being exp$(2C)$-distributed. Observe that the value of $\eta_t(x)$ for $x \in V$ and $t \in [0,\tau]$ does only depend on the sites which are in the percolation cluster containing $x$ spanned by the edges on which (at least) one clock did ring until time $\tau$. By our choice of $\tau$, each such cluster is almost surely finite. Hence, the number of transitions in the evolution of $(\eta_t(x))_{t \in [0,\tau]}$ is almost surely finite for all $x\in V$, and therefore, the evolution of $(\eta_t(x))_{t \in [0,\tau]}$ is well-defined. For $t\geq \tau$, observe that the graphical representation is Markovian and iterate the above argument to conclude. \\

In order to show that $(\eta_t)_{t \geq 0}$ is a Feller process, it remains to verify the Feller property, i.e.\@ to show that for any continuous function $f \colon \{ 0,1 \}^V \rightarrow \mathbb{R}$, we have that the mapping  
\begin{equation}\label{eq:FellerProperty}
\zeta \mapsto \E_{\zeta}[f(\eta_t)]
\end{equation}
is continuous in the starting configuration $\zeta$ for all $t\geq 0$. Using the Markov property of $(\eta_t)_{t \geq 0}$, it suffices to check this for $t \in [0,\tau]$. Moreover, since the state space $\{ 0,1 \}^V$ is equipped with the product topology and hence compact, we can assume that $f$ is a bounded function and denote its $\ell_{\infty}$-norm by $\lVert f \rVert_{\infty}$. Let $\varepsilon>0$ be fixed. For a given bounded continuous function $f$, we can choose $r=r(G,\varepsilon)$ such that \begin{equation}\label{eq:ContinuityFeller}
\sup\left\{\abs{f(\eta)-f(\zeta)} \colon \eta|_{B_r(G,o)}=\zeta|_{B_r(G,o)}\right\} < \frac{\varepsilon}{2}
\end{equation} holds where $B_r(G,o)$ denotes the ball of radius $r$ around the root $o$ of $G$ and $\eta|_{B_r(G,o)}$ is the configuration $\eta$ restricted to $B_r(G,o)$. We claim that there exists a constant $r^{\prime}=r^{\prime}(r,\tau)$ such that
\begin{equation}\label{eq:PercolationAnnulusFeller}
\P\left( \eta_t|_{B_r(G,o)}\neq\zeta_t|_{B_r(G,o)} \text{ for some } t \in [0,\tau] \ \Big| \ \eta_0|_{B_{r^{\prime}}(G,o)}=\zeta_0|_{B_{r^{\prime}}(G,o)}\right) < \frac{\varepsilon}{2 \lVert f \rVert_\infty}
\end{equation} where $(\eta_t)_{t\geq 0}$ and $(\zeta_{t})_{t \geq 0}$ perform exclusion processes according to the graphical representation using the same clocks. Note that the event in \eqref{eq:PercolationAnnulusFeller} can only occur if there exists a path of edges which are updated until time $\tau$ connecting the boundaries of the balls $B_r(G,o)$ and $B_{r^{\prime}}(G,o)$. Since we are in the subcritical regime of percolation on $G$, the probability of this event goes to $0$ for $r^{\prime}$ tending to infinity. Thus, for $\eta,\zeta \in \{ 0,1 \}^V$ with $\eta|_{B_{r^{\prime}}(G,o)}=\zeta|_{B_{r^{\prime}}(G,o)}$, we conclude that
\begin{equation*}
\abs{\E_{\eta}[f(\eta_t)] - \E_{\zeta}[f(\eta_t)]} < \varepsilon 
\end{equation*}
holds for all $t \in [0,\tau]$ by combining \eqref{eq:ContinuityFeller} and \eqref{eq:PercolationAnnulusFeller}. Hence, the exclusion process on $G$ is a Feller process. The fact that the generator of the exclusion process has indeed the form stated in \eqref{def:generatorExclusionProcess} follows by a comparison with the above graphical representation.
\end{proof}

\begin{proof}[Proof of Proposition \ref{pro:existence}]
Observe that bond percolation with parameter $p$ on a Galton--Watson tree with offspring mean $\mathfrak{m}$ gives a Galton--Watson tree of mean $\mathfrak{m}p$. Hence, $p_G=\frac{1}{\mathfrak{m}} \in (0,1)$ holds for almost every supercritical Galton--Watson tree without leaves, see \cite[Proposition 5.9]{LP:ProbOnTrees}. We conclude Proposition \ref{pro:existence} from  Proposition \ref{pro:PercolationExistence}.
\end{proof}

\begin{remark} The proof of Proposition \ref{pro:PercolationExistence} shows that interacting particle systems with only nearest neighbor interactions and uniformly bounded transition rates on locally finite graphs with strictly positive critical value for bond-percolation are well-defined Feller processes.
\end{remark}

\bibliographystyle{plain}
\bibliography{SpeedAGWT}

\begin{thebibliography}{10}

\bibitem{A:Unimodular}
David Aldous and Russell Lyons.
\newblock Processes on unimodular random networks.
\newblock {\em Electron. J. Probab.}, 12:no. 54, 1454--1508, 2007.

\bibitem{A:SpeedBiasGW}
Elie Aïdékon.
\newblock {Speed of the biased random walk on a Galton-Watson tree}.
\newblock {\em Probab. Theory Related Fields}, 159(3-4):597--617, 2014.

\bibitem{BD:InvarianceUnbounded}
M.~T. Barlow and J.-D. Deuschel.
\newblock Invariance principle for the random conductance model with unbounded
  conductances.
\newblock {\em Ann. Probab.}, 38(1):234--276, 2010.

\bibitem{C:EPrandomEnvironment}
Lincoln Chayes and Thomas~M. Liggett.
\newblock One dimensional nearest neighbor exclusion processes in inhomogeneous
  and random environments.
\newblock {\em J. Stat. Phys.}, 129(2):193--203, 2007.

\bibitem{CCGS:SpeedTree}
Dayue Chen, Peng Chen, Nina Gantert, and Dominik Schmid.
\newblock Limit theorems for the tagged particle in exclusion processes on
  regular trees.
\newblock {\em Electron. Commun. Probab.}, 24:Paper No. 2, 10, 2019.

\bibitem{EW:Ergodic}
Manfred Einsiedler and Thomas Ward.
\newblock {\em Ergodic theory with a view towards number theory}, volume 259 of
  {\em Graduate texts in mathematics}.
\newblock Springer-Verlag London, Ltd., London, 2011.

\bibitem{GMPV:RWonGW}
Nina Gantert, Sebastian Müller, Serguei Popov, and Marina Vachkovskaia.
\newblock {Random walks on Galton-Watson trees with random conductances}.
\newblock {\em Stochastic Process. Appl.}, 122(4):1652--1671, 2012.

\bibitem{GK:NetworkCompleteGraph}
Geoffrey Grimmett and Harry Kesten.
\newblock Random electrical networks on complete graphs.
\newblock {\em J. London Math. Soc. (2)}, 30(1):171--192, 1984.

\bibitem{J:Extremal}
Paul Jung.
\newblock Extremal reversible measures for the exclusion process.
\newblock {\em J. Statist. Phys.}, 112(1-2):165--191, 2003.

\bibitem{K:CLTforSEP}
Claude Kipnis.
\newblock Central limit theorems for infinite series of queues and applications
  to simple exclusion.
\newblock {\em Ann. Probab.}, 14(2):397–408, 1986.

\bibitem{KV:Additive}
Claude Kipnis and S.~R.~S. Varadhan.
\newblock {Central limit theorem for additive functionals of reversible Markov
  processes and applications to simple exclusions}.
\newblock {\em Comm. Math. Phys.}, 104(1):1--19, 1986.

\bibitem{KLO:Fluctuations}
Tomasz Komorowski, Claudio Landim, and Stefano Olla.
\newblock {\em {Fluctuations in Markov processes}}, volume 345 of {\em
  Grundlehren der mathematischen wissenschaften [Fundamental principles of
  mathematical sciences]}.
\newblock Springer, Heidelberg, 2012.

\bibitem{L:Book2}
Thomas~M. Liggett.
\newblock {\em Stochastic interacting systems: contact, voter and exclusion
  processes}, volume 324 of {\em Grundlehren der mathematischen wissenschaften
  [Fundamental principles of mathematical sciences]}.
\newblock Springer-Verlag, Berlin, 1999.

\bibitem{L:interacting-particle}
Thomas~M. Liggett.
\newblock {\em {Interacting Particle Systems}}.
\newblock Classics in mathematics. Springer-Verlag, Berlin, 2005.

\bibitem{LPP:ErgodicGalton}
Russell Lyons, Robin Pemantle, and Yuval Peres.
\newblock {Ergodic theory on Galton-Watson trees: speed of random walk and
  dimension of harmonic measure}.
\newblock {\em Ergodic Theory Dynam. Systems}, 15(3):593--619, 1995.

\bibitem{LPP:BiasedRWonGWT}
Russell Lyons, Robin Pemantle, and Yuval Peres.
\newblock {Biased random walks on Galton-Watson trees}.
\newblock {\em Probab. Theory Related Fields}, 106(2):249--264, 1996.

\bibitem{LP:ProbOnTrees}
Russell Lyons and Yuval Peres.
\newblock {\em Probability on trees and networks}, volume~42 of {\em Cambridge
  series in statistical and probabilistic mathematics}.
\newblock Cambridge University Press, New York, 2016.

\bibitem{S:TaggedParticleZd}
Ellen Saada.
\newblock A limit theorem for the position of a tagged particle in a simple
  exclusion process.
\newblock {\em Ann. Probab.}, 15(1):375–381, 1987.

\bibitem{S:ColoredBS}
Mostafa Sabri.
\newblock {The Colored Benjamini-Schramm Topology}.
\newblock 2019.

\bibitem{S:InteractionMP}
Frank Spitzer.
\newblock {Interaction of Markov processes}.
\newblock {\em Advances in Mathematics}, 5(2):246--290, October 1970.

\bibitem{Z:RWRE}
Ofer Zeitouni.
\newblock Random walks in random environment.
\newblock In {\em Lectures on probability theory and statistics}, volume 1837
  of {\em Lecture notes in math.}, pages 189--312. Springer, Berlin, 2004.

\bibitem{Z:TaggedLadder}
FuXi Zhang.
\newblock Asymptotic behavior of a tagged particle in the exclusion process on
  parallel lattices.
\newblock {\em Sci. China Math.}, 58(10):2069--2080, 2015.

\end{thebibliography}

\textbf{Acknowledgment.} We thank B\'alint T\'oth for an inspiring discussion at an early stage of this work at the 48.\@ Saint-Flour probability summer school and Alejandro Ram\'irez for helpful suggestions. Moreover, we are grateful to Noam Berger for valuable comments and pointing out the argument of Proposition \ref{pro:existence}. We thank two anonymous referees for a careful reading which helped to improve the paper significantly. The second author acknowledges the Studienstiftung des deutschen Volkes for financial support.


\end{document}